\documentclass[11pt,oneside]{amsart}
\usepackage{amssymb}
\usepackage{tikz,tikz-cd}
\usetikzlibrary{intersections}
\usepackage{graphicx}
\newcommand{\rot}[1]{#1} 
\usepackage{youngtab}
\usepackage[T1]{fontenc}\usepackage{palatino} 
\usepackage[mathscr]{euscript}
\usepackage{stmaryrd} 
\usepackage[pagebackref]{hyperref} 
\hypersetup{colorlinks,linkcolor=blue,urlcolor=blue,citecolor=blue,%
  linktocpage}  
\usepackage[abbrev,nobysame]{amsrefs} 
\usepackage[title]{appendix}   

\newtheorem{thm}{Theorem}
\newtheorem*{thm*}{Theorem}

\newtheorem*{lem*}{Lemma}

\newtheorem*{prop*}{Proposition}

\newtheorem*{cor*}{Corollary}

\newtheorem*{conj*}{Conjecture}

\newtheorem*{Donkinconjs*}{Donkin Conjectures}

\theoremstyle{definition}

\newtheorem*{defn*}{Definition}

\newtheorem*{example*}{Example}

\newtheorem*{rmk*}{Remark}

\newtheorem*{que*}{Question}

\newcommand{\C}{\mathbb{C}} 
\newcommand{\Z}{\mathbb{Z}} 

\newcommand{\Hom}{\operatorname{Hom}} 
\newcommand{\Ext}{\operatorname{Ext}} 
\newcommand{\gen}[1]{\langle #1 \rangle}
\newcommand{\ch}{\operatorname{ch}} 
\newcommand{\St}{\operatorname{St}} 
\newcommand{\soc}{\operatorname{soc}} 
\newcommand{\rad}{\operatorname{rad}} 
\newcommand{\g}{\mathfrak{g}} 
\newcommand{\im}{\operatorname{im}} 

\numberwithin{equation}{subsection} 

\allowdisplaybreaks

\title[Weyl modules for type $G_2$ in characteristic $2$]%
{A computational study of certain Weyl modules\\for type $G_2$ in
  characteristic $2$}
\author{Stephen Doty}
\email{doty@math.luc.edu}
\address{\parbox{4.75in}{Department of
    Mathematics and Statistics, Loyola University Chicago,\\ Chicago,
    Illinois 60660 USA}}

\begin{document}
\begin{abstract}\noindent
  Using the \texttt{WeylModules} \textsf{GAP} Package, we compute
  structural information about certain Weyl modules for type $G_2$ in
  characteristic $2$. This gives counterexamples to two conjectures
  stated by S.~Donkin in 1990. It also illustrates capabilities of the
  package, which can in principle be applied to Weyl modules for any
  simple, simply-connected algebraic group in any characteristic,
  subject of course to time and space limitations of computational
  resources.
\end{abstract}
\maketitle
\thispagestyle{empty}  

\section{\bf Introduction}\label{s:Intro}\noindent
There are two purposes to this paper: first, to illustrate the
computational capabilities of the author's \texttt{WeylModules}
package \cite{Doty:manual} in the context of specific examples, and
second, to give an elementary alternative computational approach to
recent counterexamples of Bendel, Nakano, Pillen, and Sobaje
\cite{BNPS:counter} obtained by different methods. 

Let $\Bbbk = \mathbb{F}_p$ be the field of $p$ elements, where $p$ is
a prime number, and let $G$ be a simple, simply-connected algebraic
group (scheme) defined and split over $\Bbbk$, with split maximal
torus $T$ and Borel subgroup $B$ determined by the negative roots.
Let $X=X(T)$ be the group of characters of $T$.
The set $X$ is partially ordered by
\[
\lambda \le \mu \iff \lambda - \mu \text{ is a sum of positive roots}.
\]
In a 1990 lecture at MSRI in Berkeley, S.~Donkin stated two
conjectures motivated by his work \cites{Donkin, Donkin:tilt} on good
filtrations of $G$-modules. (In this paper, $G$-modules are always
assumed to be rational.)

\subsection*{The Donkin conjectures}\hfill
\par A.\; The \emph{good $(p,r)$-filtration conjecture:} A $G$-module
$M$ has a good $(p,r)$-filtration if and only if $\St_r \otimes M$ has
a good filtration.

\par B.\; The \emph{tilting module conjecture:} For any $\lambda \in
X_r$, the tilting module $T(2(p^r-1)\rho + w_0\lambda)$ is isomorphic
to $Q_r(\lambda)$, as $G_r$-modules.

\medskip

See \S\ref{s:notation} for unexplained notation.  Both of the above
conjectures are in fact theorems for any $p \ge 2h-4$, where $h$ is
the Coxeter number, but for small primes they remained open until the
discovery of the first counterexamples in \cite{BNPS:counter}. The
conjectures are interrelated, and the latter was inspired by an
earlier conjecture of Humphreys and Verma. (The introduction to
\cite{BNPS:counter} gives an excellent summary of the history of these
developments.) Progress has continued in \cites{BNPS1, BNPS2,
  BNPS3,BNPS:restrict} and in particular many more counterexamples are
now known. It is interesting that the known counterexamples are
limited to $p=2,3$.

This paper is organized as follows. We define our notational
conventions in \S\ref{s:notation}. In \S\ref{s:maximal} we explain the
general ideas underlying the \texttt{WeylModules} package. The
specific case study for $\mathrm{G}_2$ in characteristic $2$ starts in
\S\ref{s:case}. The computational information obtained from the
\texttt{WeylModules} package is contained in Appendices
\ref{ss:Delta20}--\ref{ss:Delta22}.  The main theoretical conclusions
obtained from the computations are presented in \S\ref{ss:main}.
Theorem~\ref{t:TMC} gives a new proof of the counterexample to the
tilting module conjecture found in \cite{BNPS:counter}.
Theorem~\ref{t:Hom} tabulates the dimensions of the intertwining Hom
spaces between Weyl modules of highest weight bounded above by $(2,2)$.
Theorem~\ref{t:p-filt} answers an old question of Jantzen
\cite{Jan:80} in the negative and its Corollary gives a new
counterexample to the $p$-filtration conjecture.

\section{\bf Notation and terminology}\label{s:notation}\noindent
We define some general notation and terminology used in this paper.

\subsection{Notation}
We follow the notational conventions of \cite{BNPS:St}*{\S2.1}. In
particular, we let
\[
\begin{array}{lll}
  \Phi &=&  \text{root system of } (G,T).\\
  \Phi^+ &=& \text{positive roots}.\\
  \Phi^- &=& \text{negative roots}.\\
  \rho &=& \frac{1}{2} \sum_{\alpha \in \Phi^+} \alpha.\\
  \St_r &=& \Delta((p^r-1)\rho), \text{ the $r$th Steinberg module}.\\
  \St &=& \St_1 = \Delta((p-1)\rho), \text{ the Steinberg module}.\\
  X_+ &=& \{\lambda \in X(T) \mid 0 \le \gen{\lambda,\alpha^\vee},
  \text{ all simple roots } \alpha\}.\\
  L(\lambda) &=& \text{simple $G$-module of highest weight }
  \lambda \; (\lambda \in X_+).\\
  \Delta(\lambda) &=& \text{Weyl module of highest weight }
  \lambda \; (\lambda \in X_+).\\
  \nabla(\lambda) &=& \text{ind}_B^G \lambda = \text{dual Weyl module
    of highest weight } \lambda \; (\lambda \in X_+).\\
  T(\lambda) &=& \text{indecomposable tilting module of highest weight }
  \lambda \; (\lambda \in X_+).\\
  X_r &=& \{\lambda \in X_+ \mid \gen{\lambda,\alpha^\vee} < p^r,
  \text{ all simple roots } \alpha\}.\\
  G_r &=& \text{$r$th Frobenius kernel of $G$}.\\
  Q_r(\lambda) &=& \text{$G_r$-injective hull of } L(\lambda)|_{G_r},
  \text{ for } \lambda \in X_r. 
\end{array}
\]
This is nearly the same as the notation used in Jantzen's book
\cite{Jan:book}, except that our $\Delta(\lambda)$ and
$\nabla(\lambda)$ are respectively denoted by $V(\lambda)$ and
$H^0(\lambda)$ there.

\subsection{Weights and characters}
Elements of $X_+$ are \emph{dominant weights} and elements of $X=X(T)$
are \emph{weights}. A finite dimensional $G$-module $M$ is the direct
sum of its weight spaces: $M = \bigoplus_{\mu \in X} M_\mu$, where the
weight space $M_\mu$ is the eigenspace of the generalized eigenvalue
$\mu$. The (formal) character of $M$ is $\ch M = \sum_{\mu \in X}
(\dim M_\mu) e(\mu)$. For $\lambda \in X_+$, $\nabla(\lambda)$ has
simple socle and $\Delta(\lambda)$ has simple head, both of which are
isomorphic to $L(\lambda)$. For $\lambda \in X_+$, we set
\[
  \chi(\lambda) = \ch \Delta(\lambda) = \ch \nabla(\lambda); \quad
  \ell(\lambda) = \ch L(\lambda).
\]
The $\chi(\lambda)$ are given by Weyl's character formula, but the
$\ell(\lambda)$ are in general much harder to compute.

\subsection{Filtrations}\label{ss:filt}
A module $V$ has a \emph{good filtration} (resp., \emph{Weyl
filtration}) if it has a filtration with each successive quotient
isomorphic to $\nabla(\lambda)$ (resp., $\Delta(\lambda)$), for some
$\lambda \in X_+$. For any $\lambda \in X_+$, $T(\lambda)$ is the
unique (up to isomorphism) $G$-module of highest weight $\lambda$
having both a good and Weyl filtration (see \cite{Donkin:tilt}). We
have
\begin{equation}
\Delta(\lambda) \subset T(\lambda) \quad \text{and}\quad
\nabla(\lambda) \text{ is a quotient of } T(\lambda).
\end{equation}
Both $L(\lambda)$ and $T(\lambda)$ are contravariantly self-dual; that
is: $L(\lambda) = {}^\tau L(\lambda)$ and $T(\lambda) = {}^\tau
T(\lambda)$, where ${}^\tau M$ is the contravariant dual
\cite{Jan:book}*{II.2.12} of the $G$-module $M$.

The representation theory of the first Frobenius kernel $G_1$ is
equivalent to the restricted representation theory of the Lie algebra
$\text{Lie}(G)$. The region $X_1$ in $X_+$ is known as the
\emph{restricted} region. For any $r\ge 1$, any $\lambda \in X_+$ has
a unique expression of the form $\lambda = \lambda_0 + p^r \lambda_1$,
where $\lambda_0 \in X_r$, $\lambda_1 \in X_+$. Then we set
\begin{equation}
  \Delta^{(p,r)}(\lambda) = L(\lambda_0) \otimes \Delta(\lambda_1)^{(r)}, \quad
  \nabla^{(p,r)}(\lambda) = L(\lambda_0) \otimes \nabla(\lambda_1)^{(r)}
\end{equation}
where $M^{(r)}$ denotes the $r$th Frobenius twist of a $G$-module $M$.
A \emph{good} (resp., \emph{Weyl}) $(p,r)$-\emph{filtration} of a
$G$-module $M$ is a filtration with successive quotients of the form
$\nabla^{(p,r)}(\lambda)$ (resp., $\Delta^{(p,r)}(\lambda)$). In case
$r=1$, these filtrations are simply called $p$-filtrations.

\subsection{Kostant's integral form}\label{ss:Kostant}
Let $\g = \g_\C$ be the complex Lie algebra with the same root system
as $\text{Lie}(G)$. Fix a Chevalley basis
\begin{equation}
\{x_\alpha \mid \alpha \in \Phi \} \cup \{h_\alpha \mid \alpha \text{
  a simple root}\}
\end{equation}
in $\g$ and set $y_\alpha = x_{-\alpha}$ for any $\alpha \in
\Phi^+$. Let $U = U(\g)$ be the universal enveloping algebra of
$\g$. The Kostant $\Z$-form $U_\Z$ of $U$ is the $\Z$-subalgebra of
$U$ generated by all divided powers $x_\alpha^{(m)} := x_\alpha^m/m!$
($\alpha \in \Phi$, $m \ge 0$). The \emph{hyperalgebra} or
\emph{algebra of distributions} of $G$ is the algebra
\[
U_\Bbbk = U_\Z \otimes_\Z  \Bbbk . 
\]
We will always identify $x_\alpha^{(m)}$ in $U_\Z$ with its image
$x_\alpha^{(m)}\otimes 1$ in $U_\Bbbk$; similarly we identify
$\binom{h_\alpha}{m}$ with its image $\binom{h_\alpha}{m} \otimes 1$.
The category of $G$-modules is equivalent to the category of
$U_\Bbbk$-modules that admit compatible weight space decompositions.

For $\lambda \in X_+$, let $V(\lambda)$ be the simple $\g$-module of
highest weight $\lambda$. It is generated by a unique (up to scalar)
maximal vector $v_0$ (a nonzero weight vector annihilated by all
$x_\alpha$ ($\alpha \text{ simple}$)).  It is well known that
\begin{equation}
\Delta(\lambda) \cong V(\lambda)_\Z \otimes_\Z \Bbbk,   
\end{equation}
where $V(\lambda)_\Z = U_\Z v_0$ is the unique minimal admissible
lattice in $V(\lambda)$.

\section{\bf Maximal vectors}\label{s:maximal}\noindent
A nonzero weight vector in a $G$-module is \emph{maximal} if it is
invariant under the unipotent radical $U^+$ of the positive Borel
$B^+$ corresponding to the positive roots. Equivalently, a nonzero
weight vector is maximal if it is annihilated by all positive degree
divided powers of the $x_\alpha$ ($\alpha$ simple).  Maximal vectors
are also known as ``highest weight vectors''.  The core idea
underlying the \texttt{WeylModules} package \cite{Doty:manual} is that
of classifying all the maximal vectors in a given Weyl module. This is
computationally feasible only for modules of small enough dimension,
where ``small enough'' depends on the computational resources
available.  In a sense, the package automates calculations like those
of \cites{Irving,Xi}.  The \texttt{WeylModules} package builds on the
existing support for Lie algebras and their modules in \textsf{GAP}
\cite{GAP}, which in particular implements Kostant's
$\Z$-form. Calculating in the admissible lattice $V(\lambda)_\Z$, we
get information about $\Delta(\lambda)$ by reducing coefficients
modulo $p$.

\subsection{Primitive vectors}\label{ss:prim}
A nonzero weight vector $v$ in a $G$-module $M$ is
\emph{primitive} if there are submodules $M_2 \subset M_1$ of $M$ such
that $v \in M_1$ and the image of $v$ in $M_1/M_2$ is
maximal. Clearly, maximal vectors in $M$ are also primitive. See
\cite{Xi}*{\S1.1} for basic properties of primitive vectors. In
particular, if $M$ has a primitive vector of weight $\lambda$ then
$L(\lambda)$ is a composition factor of $M$.  It follows that
primitive vectors correspond bijectively with composition factors.

The main problem with primitive vectors is a lack of
uniqueness. Suppose that $v$ in $M$ is a primitive vector of weight
$\lambda$, with $M_2 \subset M_1$ as above. Let $P_1 = \gen{v}$ be the
submodule generated by $v$, and set $P_2 = M_2 \cap P_1$. Then the
image of $v$ is maximal in $P_1/P_2$. In other words, $v$ is
determined only up to some kernel.

\subsection{Intertwining morphisms}\label{ss:intertwines}
Let $\mu \in X^+$. By \cite{Jan:book}*{II.2.13} there is a functorial
isomorphism $\Hom_G(\Delta(\mu), M) \cong (M^{U^+})_\mu$, for each
rational $G$-module $M$. In other words,
\[
\Hom_G(\Delta(\mu), M) \cong \{ m \in M_\mu \mid m \text{ is a maximal
  vector or $m=0$ } \}.
\]
Thus, classifying all the maximal vectors in some fixed
$\Delta(\lambda)$ determines all the rational $G$-module homomorphisms
of the form
\begin{equation}
\Hom_G(\Delta(\mu),\Delta(\lambda))
\end{equation}
where $\mu \in X_+$ satisfies the condition $\mu \le \lambda$. This
makes it possible to compute the unique maximal submodule of
$\Delta(\lambda)$ and thus the simple character $\ell(\lambda) = \ch
L(\lambda)$. By Steinberg's tensor product theorem, it is enough to do
this for the highest weights $\lambda$ in the restricted region $X_1$,
in order to determine all the simple characters for $G$. This in turn
determines decomposition numbers of Weyl modules, submodules of Weyl
modules, and their corresponding quotients.

\subsection{Submodules and quotients}
By computing a basis for the submodule $\gen{m}$ of $\Delta(\lambda)$
generated by a maximal vector $m$ of weight $\mu$, we obtain the
character of the image of the corresponding morphism from
$\Delta(\mu)$ to $\Delta(\lambda)$. This yields structural information
on both Weyl modules; for instance, $m$ generates a copy of $L(\mu)$
in the socle of $\Delta(\lambda)$ if and only if $\ch \gen{m} =
\ell(\mu)$. Thus, we are able to compute the socle of
$\Delta(\lambda)$. Furthermore, by computing anew the maximal vectors
in the quotient module $\Delta(\lambda)/\gen{m}$ or, more generally,
$\Delta(\lambda)/S$ where $S$ is an arbitrary submodule, we can deduce
more structure. For example, taking $S = \soc \Delta(\lambda)$, the
maximal vectors of the quotient can be used to determine the second
socle $\soc^2 \Delta(\lambda)$. By iterating this, we can calculate
the full socle series of $\Delta(\lambda)$.  (See e.g.,
\cite{Humphreys:06}*{\S13.1} for definitions of socle and radical
series and their layers.)  The same method can be used to compute the
socle series of a quotient Weyl module; this requires package support
for subquotients, which is limited at present to submodules of
quotient Weyl modules.

\subsection{Extensions}
By finding a submodule $S$ such that $\Delta(\lambda)/S$ has length
two (i.e., has two linearly independent maximal vectors, which are
necessarily of distinct weights) we can find an $L(\mu)$ which appears
in the top of the radical $\rad \Delta(\lambda)$, and thus show that
$\Ext_G^1(L(\lambda), L(\mu)) \ne 0$ (see \cite{Jan:book}*{II.2.14}).
In this way, we can obtain information about extensions between simple
modules. The full radical series of $\Delta(\lambda)$ is in general
difficult to compute by the methods in the \texttt{WeylModules}
package, except in special cases.

\subsection{Ambiguous modules}\label{ss:ambig}
Following \cite{Doty:manual}, we say that a rational $G$-module is
\emph{ambiguous} if it has at least two linearly independent maximal
vectors of the same weight. Most Weyl modules of small enough highest
weight are unambiguous, in this sense. See
Appendices~\ref{ss:Delta30}, \ref{ss:Delta50}, and~\ref{ss:Delta22}
for examples of ambiguous modules.

\section{\bf Case study: $\mathrm{G}_2$ at $p=2$}\label{s:case}\noindent
From now on, we take $G$ to be of type $\mathrm{G}_2$, and we assume
that $\Bbbk = \mathbb{F}_2$ is the field of two elements.

\subsection{The root system}
Let $\varpi_1$, $\varpi_2$ be the fundamental weights for
$\mathrm{G}_2$, defined by the condition $\gen{\varpi_i,
  \alpha_j^\vee} = \delta_{i,j}$ for $i,j \in \{1,2\}$, where
$\alpha_1$, $\alpha_2$ are the simple roots, with $\alpha_1$ the short
root. Then $X = \Z \varpi_1 + \Z \varpi_2$. As usual, we identify an
element $a \varpi_1 + b\varpi_2$ in $X$ with the vector $(a,b)$ in $\Z
\times \Z$. In this notation, $X_+ = \{(a,b) \in X \mid a \ge 0, b \ge
0 \}$ and the simple roots are
\begin{equation}
\alpha_1 = (2,-1), \qquad \alpha_2 = (-3,2).
\end{equation}
Furthermore, the set of positive roots $\Phi^+$ is given by the
following table:
\[
\begin{array}{rl|l}
  \text{root} && \text{weight}\\ \hline
  \alpha_1 && (2,-1)\\
  \alpha_2 && (-3,2)\\
  \alpha_3 &=\; \alpha_2+\alpha_1 & (-1,1)\\
  \alpha_4 &=\; \alpha_2+2\alpha_1 & (1,0)\\
  \alpha_5 &=\; \alpha_2+3\alpha_1 & (3,-1)\\
  \alpha_6 &=\; 2\alpha_2 + 3\alpha_1 & (0,1)
\end{array}
\]
and we will always use this ordering of the positive roots. Let $x_i =
x_{\alpha_i}$ and $y_i = y_{\alpha_i}$ be positive and negative roots
vectors, and $h_i = h_{\alpha_i}$ the Cartan generators, in a
Chevalley basis
\begin{equation}\label{e:Chevalley}
x_1, \dots, x_6; \quad y_1, \dots, y_6; \quad h_1, h_2.
\end{equation}
These elements determine the structure of the $\Z$-form $U_\Z$, and
hence of the hyperalgebra $U_\Bbbk$, as discussed in
\S\ref{ss:Kostant}. In this paper, the Chevalley basis in
\eqref{e:Chevalley} is the one constructed by the
\texttt{ChevalleyBasis} command in \texttt{GAP}.

\subsection{Restricted Weyl modules}\label{s:restr}
The set $X_1$ of restricted weights in this case is $\{(0,0), (1,0),
(0,1), (1,1)\}$. The modules $\Delta(0,0)$ and $\Delta(1,1)$ are of
course simple, being isomorphic to the trivial and Steinberg modules,
resp. We find by a computer calculation that $\Delta(0,1)$ has a
unique maximal vector, hence must also be simple. On the other hand,
$\Delta(1,0)$ has two maximal vectors, namely $v_0$ and $y_4v_0$, and
hence is not simple. The submodule $\gen{y_4v_0}$ generated by the
latter (which has weight $(0,0)$), is isomorphic to $L(0,0) = \Bbbk$,
and the corresponding quotient $\Delta(1,0)/\gen{y_4v_0}$ has only one
maximal vector, and hence is simple (isomorphic to $L(1,0)$). Thus,
$\Delta(1,0)$ is uniserial of length two, and we have
\begin{equation}
  \Ext_G^1(L(1,0),L(0,0)) \cong \Bbbk.
\end{equation}
This also determines the simple characters of restricted highest
weight, as tabulated below,
\[
\begin{array}{c||c|c|c|c}
  \lambda & (0,0) & (1,0) & (0,1) & (1,1) \\ \hline
  \ell(\lambda)=\ch L(\lambda) &
  \chi(0,0) & \chi(1,0) - \chi(0,0) & \chi(0,1) & \chi(1,1) \\ \hline
  \dim L(\lambda) & 1 & 6 & 14 & 64
\end{array}
\]
and hence determines all the simple characters for $G$. Now that all
the simple characters are available, we can analyze the structure of
some Weyl modules of non-restricted highest weight.

\subsection{The (generalized) Schur algebra $S(\pi)$}
A theory of generalized Schur algebras was introduced by Donkin
\cites{Donkin:SA1, Donkin:SA2}, extending the construction of the
classical Schur algebras \cite{Green:80} from type A to other types of
root systems.  We will refer to them simply as Schur algebras,
dropping the ``generalized'' modifier. We are interested in studying
the structure of $\Delta(2,2)$, so we take
\[
\pi = \{\lambda \in X_+ \mid \lambda \le (2,2)\}.
\]
This set of 14 weights indexes the simple and projective modules in
the Schur algebra $S(\pi)$ defined by the saturated set
$\pi$.

We depict the points of $\pi$ below along with the alcove tiling of
the plane determined by the root system and the chosen characteristic
$2$.  In the picture, the point $-\rho$ is at the vertex of the angle
formed by the two heavy lines, which form the boundary of the shifted
dominant region $-\rho + X_+$.
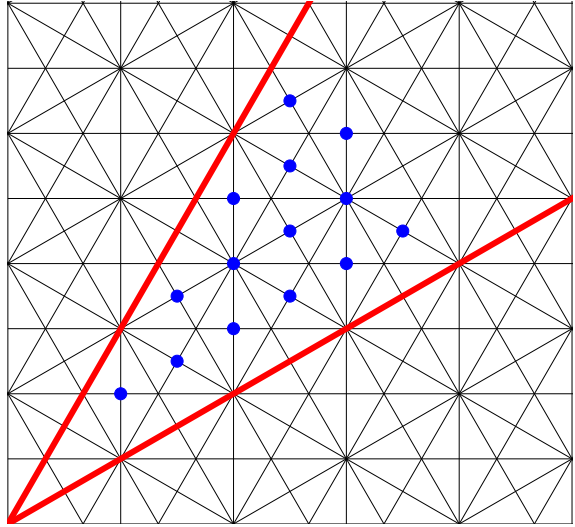
\[
\begin{tikzpicture} 
\clip(-0.0,0.0)rectangle(7.5,6.95);
\begin{scope}[scale=1.0,ultra thin]
  
\foreach\x in{-20,...,20}
{ \draw(\x,0)--++(60:40); 
  \draw(20cm+\x cm,0)--++(120:40); 
}

\foreach\x in{0,...,35}
\draw(0,0)++(60:.5*\x)++(120:.5*\x)--++(21,0); 

\foreach\y in{-4,...,10}
\draw(60:\y)++(120:\y)--++(30:25); 

\foreach\y in{1,...,15}
\draw(60:\y)++(120:\y)--++(330:25); 

\foreach\x in{0,...,8}
\draw(1.5*\x,0)--++(0,25); 

\foreach\c in{0.8660254}{
\filldraw[blue](0,0)++(60:3)++(-90:\c) circle[radius=0.08] node(a) {};
\filldraw[blue](0,0)++(60:4.5)++(-90:\c) circle[radius=0.08] node {};
\filldraw[blue](0,0)++(60:4.5)++(-90:\c*2) circle[radius=0.08] node {};
\filldraw[blue](0,0)++(60:6)++(-90:\c) circle[radius=0.08];
\filldraw[blue](0,0)++(60:6)++(-90:\c*2) circle[radius=0.08];
\filldraw[blue](0,0)++(60:6)++(-90:\c*3) circle[radius=0.08];
\filldraw[blue](0,0)++(60:7.5)++(-90:\c) circle[radius=0.08] node(b) {};
\filldraw[blue](0,0)++(60:7.5)++(-90:\c*2) circle[radius=0.08];
\filldraw[blue](0,0)++(60:7.5)++(-90:\c*3) circle[radius=0.08];
\filldraw[blue](0,0)++(60:7.5)++(-90:\c*4) circle[radius=0.08];
\filldraw[blue](0,0)++(60:9)++(-90:\c*3) circle[radius=0.08] node(c) {};
\filldraw[blue](0,0)++(60:9)++(-90:\c*4) circle[radius=0.08];
\filldraw[blue](0,0)++(60:9)++(-90:\c*5) circle[radius=0.08];
\filldraw[blue](0,0)++(60:10.5)++(-90:\c*6) circle[radius=0.08] node(d) {};
}

\draw[red, line width=0.8mm] (0,0)--++(60:40); 
\draw[red, line width=0.8mm] (0,0)--++(30:40); 
\end{scope}
\end{tikzpicture}
\]
The three dots closest to the point $-\rho$, in ascending order by
euclidean distance, are the origin $(0,0)$, $\varpi_1=(1,0)$, and
$\varpi_2=(0,1)$.
The Hasse diagram (under $\le$) of $\pi$ has the form
\[
\scriptsize
\begin{tikzpicture}[scale = 0.9, baseline={(0,-1ex/2)}]
  \node(a) at (0,0) {(0,0)};
  \node(b) at (1,0) {(1,0)};
  \node(c) at (2,0) {(0,1)};
  \node(d) at (3,0) {(2,0)};
  \node(e) at (4,0) {(1,1)};
  \node(f) at (5,0) {(3,0)};
  \node(g) at (6,0) {(0,2)};
  \node(h) at (7,0) {(2,1)};
  \node(i) at (8,0) {(4,0)};
  \node(j) at (9,0) {(1,2)};
  \node(k) at (10,0) {(3,1)};
  \node(l) at (11,-0.6) {(5,0)};
  \node(m) at (11,0.6) {(0,3)};
  \node(n) at (12,0) {(2,2)};
  \draw (a)--(b); \draw (b)--(c); \draw (c)--(d); \draw (d)--(e);
  \draw (e)--(f); \draw (f)--(g); \draw (g)--(h); \draw (h)--(i);
  \draw (i)--(j); \draw (j)--(k);
  \draw (k)--(l);
  \draw (k)--(m);
  \draw (l)--(n);
  \draw (m)--(n);
\end{tikzpicture} .
\]
The set $\pi$ breaks up into just two linkage classes: $\{(1,1),
(3,1)\}$, and everything else. Hence, the Schur algebra $S(\pi)$ has
exactly two blocks: the Steinberg block (which has the two weights
$(1,1)$ and $(3,1)$; of course $(1,1)$ is the highest weight of the
Steinberg module) and the trivial block. As all the points of $\pi$
are located at vertices of the alcoves they meet, Jantzen's
translation principle carries zero information for this case study in
characteristic~$2$.

\subsection{Main results}\label{ss:main}
We now give our main conclusions from the case study.  We begin with a
new computational proof of a counterexample to the tilting module
conjecture.

\begin{thm}[\cite{BNPS:counter}]\label{t:TMC}
  The $G$-socle of the tilting module $T(2,2)$ is not isomorphic to
  $L(0,0) \cong \Bbbk$. Hence, $T(2,2)$ is not isomorphic to
  $Q_1(0,0)$ as $G_1$-modules. In other words, $T(2,2)$ is a
  counterexample to Donkin's tilting module conjecture.
\end{thm}

\begin{proof}
The \texttt{WeylModules} package computes (unique up to scalar
multiple) maximal vectors of weight $(0,1)$ and $(0,0)$ in
$\Delta(2,2)$; these vectors are listed in
Appendix~\ref{ss:Delta22}. It follows that
$\Hom_G(\Delta(\mu),\Delta(2,2)) \cong \Bbbk$ for $\mu = (0,1)$ and
$(0,0)$. As $\Delta(0,1) = L(0,1)$ and $\Delta(0,0) = L(0,0)$ are
simple modules, it follows that the $G$-socle of $\Delta(2,2)$
contains at least $L(0,0) \oplus L(0,1)$.  Furthermore, since
$\Delta(2,2) \subset T(2,2)$, it follows that $\soc_G \Delta(2,2)
\subset \soc_G T(2,2)$. Thus $\soc_G T(2,2) \not\cong \Bbbk$, proving
the first claim.  Now assume that $T(2,2)|_{G_1} \cong Q_1(0,0)$. Then
\[
\soc_{G_1} T(2,2) \cong \Bbbk
\]
by the definition of $Q_1(0,0)$ as the injective hull of $L(0,0) =
\Bbbk$ in the category of rational $G_1$-modules.  But the inclusion
$G_1 \subset G$ implies that
\[
\soc_{G} T(2,2) \subset \soc_{G_1} T(2,2)
\]
which violates the first claim. The proof is complete.
\end{proof}

\begin{rmk*}
Using the \texttt{WeylModules} package, we have verified the more
precise fact that
\[
\soc_G \Delta(2,2) \cong L(0,1) \oplus L(0,0).
\]
See Appendix~\ref{ss:Delta22} for the complete socle series of $\Delta(2,2)$.
\end{rmk*}

The next result classifies the dimensions of the spaces of
intertwining morphisms between Weyl modules of highest weight
belonging to $\pi$.

\begin{thm}\label{t:Hom}
Let $\pi_0$ be the set of highest weights $\lambda$ indexing simples
in the trivial block of $S(\pi)$, and set $\pi_1 = \pi \setminus \pi_0
= \{(1,1),(3,1)\}$, the Steinberg block.  Then the intertwining
dimensions $\dim_\Bbbk \Hom_G(\Delta(\mu),\Delta(\lambda))$ for any
$\lambda \le \mu$ in $\pi_0$ are:
\[
\begin{array}{c|cccccccccccc}
  \lambda \backslash \mu  &
  \rot{00}&\rot{10}&\rot{01}&\rot{20}&\rot{30}&\rot{02}&\rot{21}&\rot{40}
  &\rot{12}&\rot{50}&\rot{03}&\rot{22}\\ \hline
00&  1&\cdot&\cdot&\cdot&\cdot&\cdot&\cdot&\cdot&\cdot&\cdot&\cdot&\cdot \\
10&  1&1&\cdot&\cdot&\cdot&\cdot&\cdot&\cdot&\cdot&\cdot&\cdot&\cdot \\
01&  0&0&1&\cdot&\cdot&\cdot&\cdot&\cdot&\cdot&\cdot&\cdot&\cdot\\
20&  0&1&1&1&\cdot&\cdot&\cdot&\cdot&\cdot&\cdot&\cdot&\cdot\\
30&  1&1&1&2&1&\cdot&\cdot&\cdot&\cdot&\cdot&\cdot&\cdot\\
02&  0&0&1&1&1&1&\cdot&\cdot&\cdot&\cdot&\cdot&\cdot\\
21&  0&0&1&1&1&1&1&\cdot&\cdot&\cdot&\cdot&\cdot\\
40&  1&1&0&1&1&0&1&1&\cdot&\cdot&\cdot&\cdot\\
12&  0&1&0&1&1&0&1&1&1&\cdot&\cdot&\cdot\\
50&  1&1&1&1&1&1&1&2&1&1&\cdot&\cdot\\
03&  1&1&0&0&1&1&0&1&0&0&1&\cdot\\
22&  1&1&1&1&1&1&1&2&1&1&1&1
\end{array}
\]
and the same tabulation for any $\lambda \le \mu$ in $\pi_1$ is given by:
\[
\begin{array}{c|cc}
  \lambda \backslash \mu  &
  \rot{11}&\rot{31} \\ \hline
  11&  1&\cdot\\
  31&  1&1
\end{array}
\]
where in both tables $\lambda$ indexes the rows, $\mu$ indexes the
columns, and rows and columns are ordered by a linear order compatible
with $\le$. The modules $\Delta(3,0)$, $\Delta(5,0)$, and
$\Delta(2,2)$ are all ambiguous.  To enhance readability, the
punctuation has been omitted in the row and column labels (e.g.,
$(2,2)$ is written as $22$, etc).
\end{thm}

\begin{proof}
This is proved by a computer calculation of all the maximal vectors in
the Weyl modules $\Delta(\lambda)$ as $\lambda$ varies over $\pi$. The
computed maximal vectors are listed in \S\ref{s:restr} and in
the appendices.
\end{proof}

\begin{thm}\label{t:p-filt}
  The dual Weyl module $\nabla(0,2)$ does not admit a good
  $p$-filtration.
\end{thm}

\begin{proof}
(Compare with the proof of \cite{BNPS:counter}*{Thm.~3.5.1}.)
Assume for a contradiction that $\nabla(0,2)$ has a good
$p$-filtration (see \S\ref{ss:filt})
\[
0 = F_0 \subset F_1 \subset \cdots \subset F_{n-1} \subset
F_n = \nabla(0,2).
\]
The socle series of $\Delta(0,2)$ is computed in
Appendix~\ref{ss:Delta02}.  The radical series of $\nabla(0,2)$ is
obtained by turning the socle series of $\Delta(0,2)$ upside down, so
we know the radical series of $\nabla(0,2)$. Thus, by the information
in Appendix~\ref{ss:Delta02}, we must have
\[
\text{$F_n/F_{n-1} \cong L(0,1)$ and $F_{n-1}/F_{n-2} \cong
  \nabla(\mu)^{(1)}$,}
\]
for some dominant weight $\mu$ such that $L(1,0)^{(1)}$ is the
$G$-head of $\nabla(\mu)^{(1)}$. This implies that $p\mu = 2\mu \le
(0,2)$, so $\mu \le (0,1)$.  Hence $\mu$ must be one of the weights
$(0,0)$, $(1,0)$, or $(0,1)$. This forces $\mu = (1,0)$, as we know
that $\nabla(0,0) = L(0,0)$ and $\nabla(0,1) = L(0,1)$.  But
$\nabla(1,0) \not\cong L(1,0)$, so $L(1,0)^{(1)}$ is not the $G$-head
of $\nabla(1,0)^{(1)}$. This is a contradiction.
\end{proof}

The module $\nabla(0,2)$ answers an old question of Jantzen
\cite{Jan:80} (whether all $\nabla(\lambda)$ have a good
$p$-filtration, or equivalently, whether all Weyl modules have a Weyl
$p$-filtration) in the negative. It also provides a new counterexample
to Donkin's $(p,r)$-filtration conjecture stated in
\S\ref{s:Intro}. 

\begin{cor*}
The module $\nabla(0,2)$ fails to satisfy the $(p,r)$-filtration
conjecture.
\end{cor*}

\begin{proof}
By the general result in \cite{Donkin} or \cite{Mathieu}, $\St \otimes
\nabla(0,2)$ has a good filtration. But, according to
Theorem~\ref{t:p-filt}, $\nabla(0,2)$ fails to have a good
$p$-filtration.
\end{proof}

See the appendices for further computational results. For example, we
calculate the composition factors of the socle series of all the Weyl
modules of non-restricted highest weight in $\pi$. We also compute the
complete submodule lattice for the Weyl modules $\Delta(2,0)$,
$\Delta(0,2)$, and $\Delta(0,3)$ in Appendices~\ref{ss:Delta20},
\ref{ss:Delta02}, and \ref{ss:Delta03} respectively. This in turn
determines some extension groups between simple modules.

\section*{\bf APPENDICES}\noindent
The structure of the Weyl modules of restricted highest weight was
determined in \S\ref{s:restr} above.  In the following appendices, we
compile structural details about the $\Delta(\lambda)$ such that
$\lambda$ is a non-restricted weight in the set $\pi = \{ \lambda \in
X_+ \mid \lambda \le (2,2) \}$. In particular, we calculate:
\begin{itemize}
\item All the maximal vectors in $\Delta(\lambda)$.
\item The socle series of $\Delta(\lambda)$.
\item The character of the image of morphisms $\Delta(\mu) \to
  \Delta(\lambda)$ determined by maximal vectors of weight $\mu$ in
  $\Delta(\lambda)$.
\item The socle series of the cokernel of many of these morphisms.
\item Related structural information.
\end{itemize}
It is hoped that such information will be useful in determining the
structure of all the Weyl modules in the trivial block
$S(\pi_0)$. Ultimately, one would like to describe this algebra by
quiver and relations, as in Ringel's appendix to \cite{BDM1}. This
would settle an outstanding question about the algebra: what is the
structure of $T(2,2)$ as a $G_1$-module? It is clear from
Theorem~\ref{t:TMC} that $T(2,2)|_{G_1}$ contains at least $Q_1(0,0)
\oplus Q_1(0,1)$. Is it equal to that module, or is the containment
proper?


\appendix
\renewcommand{\thesection}{\Alph{section}} 

\section{\bf Structure of $\Delta(2,0)$}\label{ss:Delta20}\noindent 
The highest weights of the composition factors of the socle series
layers of $\Delta(2,0)$ are:
\[
\begin{array}{c|c}
\text{socle layer} & \text{highest weights}\\ \hline
3 & (2, 0)\\
2 & (0, 0)\\
1 & \phantom{.}(0, 1),\; (1, 0) .
\end{array}
\]
From now on, we will always describe a socle series by such a table.  This
Weyl module has three maximal vectors, as listed below:
\begin{align*}
m(2, 0) = v_0, \quad
m(0, 1) = y_1 v_0, \quad
m(1, 0) = y_4 v_0 .
\end{align*}
The notation $m(\lambda)$ here stands for a maximal vector of
weight $\lambda$. If there are two or more linearly independent
maximal vectors of weight $\lambda$ (the module is ambiguous), they
will be denoted by $m_1(\lambda)$, $m_2(\lambda)$, and so on.


\subsection{The morphism $\Delta(0,1) \to \Delta(2,0)$}
The image of the morphism sending the generator of $\Delta(0,1)$ onto
$m(0,1)$ is isomorphic to $L(0,1)$. The highest weights of the
composition factors of the socle series layers of the cokernel $Q'$ of
this morphism are:
\[
\begin{array}{c|c}
  \text{radical layer} & \text{highest weight}\\ \hline
  1 & (2,0) \\
  2 & (0,0)\\
  3 & \phantom{.}(1,0). 
\end{array}
\]
This is a uniserial module of length three. 

\subsection{The morphism $\Delta(1,0) \to \Delta(2,0)$}
The image of the morphism sending the generator of $\Delta(1,0)$ onto
$m(1,0)$ is isomorphic to $L(1,0)$.  The highest weights of the
composition factors of the socle series layers of the cokernel $Q$ of
this morphism are:
\[
\begin{array}{c|c}
\text{socle layer} & \text{highest weights}\\ \hline
2 & (2, 0)\\
1 & \phantom{.}(0, 1),\; (0, 0) .
\end{array}
\]
This structure shows that $L(0,1)$ and $L(0,0)$ are both top
factors of the radical of $\Delta(2,0)$.

\subsection{The module $Q'$}\label{ss:Q'}
The maximal vectors in the cokernel $Q$ are given below:
\begin{align*}
m(2, 0) = v_0, \quad
m(0, 1) = y_1 v_0, \quad
m(0, 0) = y_1 y_6 v_0+y_3 y_5 v_0+y_4^{(2)} v_0 .
\end{align*}
The pre-image of $m(0,0)$ in $\Delta(2,0)$ generates a submodule
$\gen{m(0,0)}$ realizing a non-split extension of $L(0,0)$ by
$L(1,0)$, which is isomorphic to $\nabla(1,0)$. In other words, a copy
of $\nabla(1,)$ embeds in $\Delta(2,0)$.  The corresponding quotient
$Q' = \Delta(2,0)/\gen{m(0,0)}$ is of length two, and realizes a
non-split extension of $L(2,0)$ by $L(0,1)$.

\subsection{Radical series and structure diagram}
It follows from the preceding subsection that the highest weights of
the composition factors of the radical series layers of $\Delta(2,0)$
must be of the form:
\[
\begin{array}{c|c}
  \text{radical layer} & \text{highest weight}\\ \hline
  1 & (2,0) \\
  2 & (0,1), \; (0,0) \\
  3 & \phantom{.}(1,0). 
\end{array}
\]
This demonstrates that the module $\Delta(2,0)$ is not rigid.

The structure diagram of $\Delta(2,0)$, in the sense of Alperin
diagrams (see \cite{Alperin}), has the form:
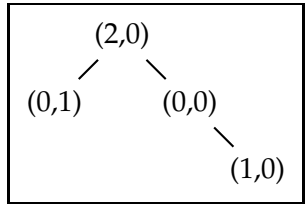
\[
\boxed{
\begin{tikzpicture}[scale = 0.9,thick, baseline={(0,-1ex/2)}]
  \node(a) at (0,0) {(2,0)};
  \node(b) at (-1,-1) {(0,1)};
  \node(c) at (1,-1) {(0,0)};
  \node(d) at (2,-2) {(1,0)};
  \draw (a)--(b); 
  \draw (a)--(c);
  \draw (c)--(d);
\end{tikzpicture}
}
\]
where the nodes $\mu$ are labeled by the highest weights of simple
composition factors $L(\mu)$, and edges depict non-split extensions.
From this we deduce that
\begin{equation}
  \Ext^1_G(L(2,0),L(0,1)) \cong \Bbbk \cong \Ext^1_G(L(2,0),L(0,0))
\end{equation}
and, furthermore, there are no other non-split extensions of $L(2,0)$
by a simple $L(\mu)$ of highest weight $\mu \le (2,0)$.

\subsection{Remark}\label{r:Q'}
The module $Q'$ constructed in \S\ref{ss:Q'} above is isomorphic to
the contravariant dual of a module, denoted by $M$ in
\cite{BNPS:counter} and constructed very differently, which plays an
important role in the calculations there.  We note that $M={}^\tau Q'$
does not have a good $p$-filtration, in the sense of
\S\ref{ss:filt}. Indeed, if it did then necessarily its socle $L(2,0)$
would have to be of the form $L(\lambda_0) \otimes
\nabla(\lambda_1)^{(1)}$, for some $\lambda = \lambda_0 + p
\lambda_1$, where $\lambda_0 \in X_1$. But $L(1,0) \not\cong
\nabla(1,0)$ implies that $L(2,0) \cong L(1,0)^{(1)}$ is not
isomorphic to $\nabla(1,0)^{(1)}$. The authors of the cited paper
prove that $\St \otimes M$ has a good filtration. It follows (as they
observe) that $M$ is a counterexample to the good $(p,r)$-filtration
conjecture stated in \S\ref{s:Intro}.

\section{\bf Structure of $\Delta(3,0)$}\label{ss:Delta30}\noindent
This interesting Weyl module illustrates many complexities.  We begin
by computing the highest weights of the composition factors of its
socle series layers:
\[
\begin{array}{c|c}
\text{socle layer} & \text{highest weights}\\ \hline
6 & (3, 0)\\
5 & (1, 0)\\
4 & (0, 0)\\
3 & (2, 0)\\
2 & (2, 0),\; (0, 0)\\
1 & \phantom{.}(0, 1),\; (1, 0),\; (0, 0) .
\end{array}
\]
The maximal vectors in $\Delta(3,0)$ are listed below:
\begin{gather*}
m(3, 0) = v_0, \quad
m_1(2, 0) = y_4 v_0, \quad
m_2(2, 0) = y_1 y_3 v_0, \quad
m(0, 1) = y_1 y_4 v_0, \\
m(1, 0) = y_1 y_3 y_4 v_0+y_1 y_6 v_0+y_3 y_5 v_0, \\
m(0, 0) = y_1 y_4 y_6 v_0+y_3 y_4 y_5 v_0+y_4^{(3)} v_0 .
\end{gather*}
There are two linearly independent maximal vectors of weight $(2,0)$,
so $\Delta(3,0)$ is ambiguous, in the sense defined in
\S\ref{ss:ambig}.

\subsection{The morphism $\varphi_1: \Delta(2,0) \to \Delta(3,0)$}
For $j=1,2$ let $\varphi_j$ be the morphism $\Delta(2,0) \to
\Delta(3,0)$ defined by sending the generator of $\Delta(2,0)$ onto
$m_j(2,0)$.  The character of the image of $\varphi_1$ is
$\ell(2,0)+\ell(0,1)+\ell(0,0)$.  By Appendix~\ref{ss:Delta20}, this
determines the structure of $\im(\varphi_1)$, as it is isomorphic with
a quotient of $\Delta(2,0)$.  The highest weights of the composition
factors of the socle series layers of the cokernel of $\varphi_1$
(isomorphic to $Q_1 = \Delta(3,0)/\im(\varphi_1)$) are:
\[
\begin{array}{c|c}
\text{socle layer} & \text{highest weights}\\ \hline
6 & (3, 0)\\
5 & (1, 0)\\
4 & (0, 0)\\
3 & (2, 0)\\
2 & (0, 0)\\
1 & \phantom{.}(1, 0).
\end{array}
\]
The quotient $Q_1$ is uniserial. We will see in Appendix~\ref{ss:Delta02}
below that it appears as a homomorphic image of the (unique up to
scalar multiple) morphism from $\Delta(3,0)$ into $\Delta(0,2)$.

\subsection{The morphism $\varphi_2: \Delta(2,0) \to \Delta(3,0)$}
The character of the image of $\varphi_2$ is
$\ell(2,0)+\ell(1,0)+\ell(0,0)$. By Appendix~\ref{ss:Delta20}, this
determines its structure.  The highest weights of the composition
factors of the socle series layers of the cokernel of $\varphi_2$ are:
\[
\begin{array}{c|c}
\text{socle layer} & \text{highest weights}\\ \hline
5 & (3, 0)\\
4 & (1, 0)\\
3 & (0, 0)\\
2 & (2, 0)\\
1 & \phantom{.}(0, 1),\; (0, 0).
\end{array}
\]

\subsection{The morphism $\varphi_1+\varphi_2: \Delta(2,0) \to \Delta(3,0)$}
The character of the image of $\varphi_1+\varphi_2$ is
$\ell(2,0)+\ell(1,0)+\ell(0,1)+\ell(0,0)$. It follows that
$\im(\varphi_1+\varphi_2)$ is isomorphic to $\Delta(2,0)$; that is, an
isomorphic copy of $\Delta(2,0)$ embeds in $\Delta(3,0)$.
%
%
The highest weights of the composition factors of the socle series
layers of the cokernel $Q_3$ of $\varphi_1+\varphi_2$ give a very
different picture:
\[
\begin{array}{c|c}
\text{socle layer} & \text{highest weights}\\ \hline
3 & (3, 0)\\
2 & (2, 0),\; (1, 0)\\
1 & \phantom{.}(0, 0),\; (0, 0).
\end{array}
\]
This quotient is an ambiguous module.  Its socle is isomorphic to the
direct sum of two copies of the trivial module. We record the fact
that
\begin{equation}
\dim_\Bbbk \Hom(\Bbbk, Q_3) = \dim_\Bbbk \Hom(\Delta(0,0), Q_3) = 2.
\end{equation}
The structure of the cokernel $Q_3$ suggests that a copy of $L(2,0)$
appears as a top factor of $\rad \Delta(3,0)$, which would imply that
\begin{equation}
\Ext_G^1(L(3,0),L(2,0)) \ne 0.
\end{equation}
We have verified the truth of this suspicion as follows.  First, we
computed the maximal vectors in $Q_3$, which are given by:
\begin{align*}
v_0,\;
y_1 y_3 v_0, \;
y_1 y_6 v_0+y_4^{(2)} v_0, \;
y_4^{(3)} v_0,\; y_5 y_6 v_0 
\end{align*}
of respective weights $(3,0)$, $(2,0)$, $(1,0)$, $(0,0)$, $(0,0)$.
Then we constructed the quotient $\Delta(3,0)/S$, where
\[
  S = \gen{m_1(2,0)+m_2(2,0),y_1 y_6 v_0+y_4^{(2)} v_0,
    y_4^{(3)} v_0, y_5 y_6 v_0}
\]
is the submodule generated by $m_1(2,0)+m_2(2,0)$ along with the
preimages of the last three maximal vectors computed above.  It turns
out that $\Delta(3,0)/S$ is uniserial, with two composition factors
$L(3,0)$ and $L(2,0)$, justifying the claim. A similar quotient
verifies that $L(1,0)$ appears as a top factor of $\rad \Delta(3,0)$,
and hence that
\begin{equation}
\Ext_G^1(L(3,0),L(1,0)) \ne 0.
\end{equation}
Another way of constructing such an extension is by taking the
quotient by the fourth socle. At this point we have proved that
$\Delta(3,0)$ is not rigid, as it has at least two composition factors
in the top of its radical.

\subsection{The morphism $\Delta(0,1) \to \Delta(3,0)$}
The image of the morphism sending the generator of $\Delta(0,1)$ onto
$m(0,1)$ is isomorphic to $L(0,1) = \Delta(0,1)$. The highest weights
of the socle series layers of the cokernel of this morphism are:
\[
\begin{array}{c|c}
\text{socle layer} & \text{highest weights}\\ \hline
6 & (3, 0)\\
5 & (1, 0)\\
4 & (0, 0)\\
3 & (2, 0)\\
2 & (2, 0),\; (0, 0)\\
1 & \phantom{.}(1, 0),\; (0, 0).
\end{array}
\]
This is what would have been predicted by deleting the factor of
highest weight $(0,1)$ from the socle series layers of $\Delta(3,0)$.

\subsection{The morphism $\Delta(1,0) \to \Delta(3,0)$}
The image of the morphism sending the generator of $\Delta(1,0)$ onto
$m(1,0)$ is isomorphic to $L(1,0)$. The highest weights
of the socle series layers of the cokernel of this morphism are:
\[
\begin{array}{c|c}
\text{socle layer} & \text{highest weights}\\ \hline
5 & (3, 0)\\
4 & (1, 0)\\
3 & (0, 0)\\
2 & (2, 0),\; (2, 0)\\
1 & \phantom{.}(0, 1),\; (0, 0),\; (0, 0).
\end{array}
\]
This is slightly surprising, as we picked up a second factor of $L(0,0)$ in
the socle.

\subsection{The morphism $\Delta(0,0) \to \Delta(3,0)$}
The image of the morphism sending the generator of $\Delta(0,0) \cong
\Bbbk$ onto $m(0,0)$ is isomorphic to $L(0,0) \cong \Bbbk$. The
highest weights of the socle series layers of the cokernel of this
morphism are:
\[
\begin{array}{c|c}
\text{socle layer} & \text{highest weights}\\ \hline
6 & (3, 0)\\
5 & (1, 0)\\
4 & (0, 0)\\
3 & (2, 0)\\
2 & (2, 0),\; (0, 0)\\
1 & \phantom{.}(0, 1),\; (1, 0).
\end{array}
\]
This is again compatible with deleting $L(0,0)$ from the socle layer
in the socle series of $\Delta(3,0)$.

Next, we examine some other interesting quotients of $\Delta(3,0)$. 

\subsection{The quotient $Q_4$}
The module $Q_4 = \Delta(3,0)/\gen{m(0,1), m(1,0)}$ is another
interesting quotient. The highest weights of the composition factors
of the socle series layers are:
\[
\begin{array}{c|c}
\text{socle layer} & \text{highest weights}\\ \hline
5 & (3, 0)\\
4 & (1, 0)\\
3 & (0, 0)\\
2 & (2, 0),\; (2, 0)\\
1 & \phantom{.}(0, 0),\; (0, 0) .
\end{array}
\]
From this we deduce that
\begin{equation}
  \dim_\Bbbk \Hom_G(\Delta(1,0)^{(1)}, Q_4) = 2.
\end{equation}
More precisely, there is an embedding of a direct sum of two copies of
$\Delta(1,0)^{(1)}$ in $Q_4$, and $\soc_G Q_4 \cong L(0,0) \oplus L(0,0)$.
This is yet another ambiguous module.

\subsection{The quotient $Q_5$}
The module $Q_5 = \Delta(3,0)/\gen{m(1,0), m(0,0)}$ is also
interesting.  The highest weights of the composition factors of the
socle series layers for $Q_5$ are:
\[
\begin{array}{c|c}
\text{socle layer} & \text{highest weights}\\ \hline
5 & (3, 0)\\
4 & (1, 0)\\
3 & (0, 0)\\
2 & (2, 0),\; (2, 0)\\
1 & \phantom{.}(0, 1),\; (0, 0) .
\end{array}
\]

These calculations reveal various aspects of the structure of
$\Delta(3,0)$, but a complete understanding of the submodule structure
of this module seems rather elusive at this point.

\section{\bf Structure of $\Delta(0,2)$}\label{ss:Delta02}\noindent
The highest weights of the composition factors of the socle series
layers of this Weyl module are given by:
\[
\begin{array}{c|c}
\text{socle layer} & \text{highest weights}\\ \hline
6 & (0, 2)\\
5 & (3, 0)\\
4 & (1, 0)\\
3 & (0, 0)\\
2 & (2, 0)\\
1 & \phantom{.}(0, 1) .
\end{array}
\]
This shows that $\Delta(0,2)$ is uniserial. According to
\cite{BNPS:counter}, this independently verifies an observation of
H.H.~Andersen. The maximal vectors in $\Delta(0,2)$ are:
\begin{gather*}
m(0, 2) = v_0,\; m(3, 0) = y_2 v_0,\\  m(2, 0) = y_2 y_4 v_0,\;
m(0, 1) = y_2 y_5 v_0+y_3 y_4 v_0 .
\end{gather*}
The module $\Delta(0,2)$ is rigid and unambiguous, and has the obvious
linear Alperin diagram corresponding to its socle series layers.

\subsection{The morphisms $\Delta(\lambda) \to \Delta(0,2)$ for
  $\lambda = (3,0)$, $(2,0)$, and $(0,1)$ }
In each case, the image of the morphism sending the generator of
$\Delta(\lambda)$ onto $m(\lambda)$ is uniserial. Its character can be
read off from the above socle series. The cokernel of each morphism is
also uniserial, and its structure can also be read off from the above
socle series. We note that when $\lambda=(2,0)$, the image of the
morphism is isomorphic to the module ${}^\tau M$ discussed in
Remark~\ref{r:Q'}.

\section{\bf Structure of $\Delta(2,1)$}\label{ss:Delta21}\noindent
The highest weights of the composition factors of the socle series
layers of $\Delta(2,1)$ are given by:
\[
\begin{array}{c|c}
\text{socle layer} & \text{highest weights}\\ \hline
10 & (2, 1)\\
9 & (0, 1),\; (0, 0)\\
8 & (2, 0)\\
7 & (0, 0)\\
6 & (0, 2),\; (1, 0)\\
5 & (3, 0)\\
4 & (1, 0)\\
3 & (0, 0)\\
2 & (2, 0)\\
1 & \phantom{.}(0, 1).
\end{array}
\]
The maximal vectors in $\Delta(2,1)$ are listed below:
\begin{gather*}
m(2, 1) = v_0, \quad
m(0, 2) = y_1 v_0, \quad
m(3, 0) = y_1 y_2 v_0+y_3 v_0, \\
m(2, 0) = y_1 y_2 y_4 v_0+y_2 y_5 v_0+y_3 y_4 v_0, \quad
m(0, 1) = y_1 y_2 y_5 v_0+y_1 y_3 y_4 v_0 .
\end{gather*}
This Weyl module is not ambiguous.

\subsection{The morphism $\Delta(0,2) \to \Delta(2,1)$}
The image of the morphism sending the generator of $\Delta(0,2)$ onto
$m(0,2)$ has character $\ell(0, 2) + \ell(3, 0) + \ell(2, 0) + \ell(0,
1) + \ell(1, 0) + \ell(0, 0)$. It follows from the structure of
$\Delta(0,2)$ that the image of this morphism is isomorphic to
$\Delta(0,2)$. In other words, a copy of $\Delta(0,2)$ embeds in
$\Delta(2,1)$.  The highest weights of the composition factors of the
socle series layers for the cokernel of this morphism are:
\[
\begin{array}{c|c}
\text{socle layer} & \text{highest weights}\\ \hline
5 & (2, 1)\\
4 & (0, 1),\; (0, 0)\\
3 & (2, 0)\\
2 & (0, 0)\\
1 & \phantom{.}(1, 0).
\end{array}
\]

\subsection{The morphism $\Delta(3,0) \to \Delta(2,1)$}
The image of the morphism sending the generator of $\Delta(3,0)$ onto
$m(3,0)$ has character $\ell(3, 0) + \ell(2, 0) + \ell(0, 1) + \ell(1,
0) + \ell(0, 0)$.  This image is a uniserial module. The highest
weights of the composition factors of the socle series layers of the
cokernel of this morphism are:
\[
\begin{array}{c|c}
\text{socle layer} & \text{highest weights}\\ \hline
5 & (2, 1)\\
4 & (0, 1),\; (0, 0)\\
3 & (2, 0)\\
2 & (0, 0)\\
1 & \phantom{.}(0, 2),\; (1, 0).
\end{array}
\]

\subsection{The morphism $\Delta(2,0) \to \Delta(2,1)$}
The image of the morphism sending the generator of $\Delta(2,0)$ onto
$m(2,0)$ has character $\ell(2, 0) + \ell(0, 1)$. This image is
isomorphic to the module ${}^\tau M$ which was discussed in
Remark~\ref{r:Q'}.  The highest weights of the composition
factors of the socle series layers of the cokernel of this morphism
are:
\[
\begin{array}{c|c}
\text{socle layer} & \text{highest weights}\\ \hline
8 & (2, 1)\\
7 & (0, 1),\; (0, 0)\\
6 & (2, 0)\\
5 & (0, 0)\\
4 & (0, 2),\; (1, 0)\\
3 & (3, 0)\\
2 & (1, 0)\\
1 & \phantom{.}(0, 0).
\end{array}
\]

\subsection{The morphism $\Delta(0,1) \to \Delta(2,1)$}
The image of the morphism sending the generator of $\Delta(0,1)$ onto
$m(0,1)$ is isomorphic to $L(0,1)$. The socle series layers for its
cokernel can be deduced from those for $\Delta(2,1)$ itself.

\subsection{The radical of $\Delta(2,1)$}
The module $\Delta(2,1)$ contains a primitive vector
\[
z = y_1y_2y_5v_0 + y_4^{(2)}v_0
\]
of weight $(0,1)$ that generates the simple module
$L(0,1)$ in the quotient of $\Delta(2,1)$ by its eighth socle.
Consider the quotient $Q_1 = \Delta(2,1)/\gen{z}$ by the submodule
generated by the primitive vector $z$. The module $Q_1$ has three
composition factors and three maximal vectors, and has an Alperin
diagram of the form
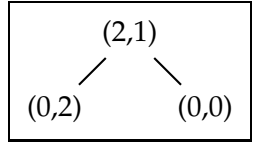
\[
\boxed{
\begin{tikzpicture}[scale = 1,thick, baseline={(0,-1ex/2)}]
  \node(a) at (0,0) {(2,1)};
  \node(b) at (-1,-1) {(0,2)};
  \node(c) at (1,-1) {(0,0)};
  \draw (a)--(b); 
  \draw (a)--(c);
\end{tikzpicture}
}
\]
which shows in particular that $L(0,2)$ is a top composition factor in
the radical of $\Delta(2,1)$. We conclude that $\Delta(2,1)$ is not
rigid. Furthermore, we have
\begin{equation}
\Ext_G^1(L(2,1),L(0,2)) \ne 0 \quad \text{and} \quad
\Ext_G^1(L(2,1),L(0,0)) \ne 0.
\end{equation}

There is another primitive vector $z' = y_1y_2y_5y_6v_0$ generating a
top composition factor isomorphic to $L(0,0)$ in the ninth socle of
$\Delta(2,1)$.  The corresponding quotient $Q_2 =
\Delta(2,1)/\gen{z'}$ is a highest weight module with two composition
factors, $L(2,1)$ and $L(0,1)$, which shows that
\begin{equation}
\Ext_G^1(L(2,1),L(0,1)) \ne 0.
\end{equation}
This proves that the second radical layer of $\Delta(2,1)$ has at
least three composition factors.

\section{\bf Structure of $\Delta(4,0)$}\label{ss:Delta40}\noindent
The socle series layers of this Weyl module have composition factors
of the following highest weights:
\[
\begin{array}{c|c}
\text{socle layer} & \text{highest weights}\\ \hline
9 & (4, 0)\\
8 & (0, 0)\\
7 & (0, 2),\; (2, 0)\\
6 & (3, 0)\\
5 & (2, 1),\; (1, 0)\\
4 & (0, 1),\; (0, 0)\\
3 & (2, 0)\\
2 & (0, 0)\\
1 & \phantom{.}(1, 0),\; (0, 0).
\end{array}
\]
The maximal vectors in $\Delta(4,0)$ are listed below:
\begin{gather*}
m(4, 0) = v_0, \quad
m(2, 1) = y_1 v_0, \quad
m(3, 0) = y_4 v_0, \quad 
m(2, 0) = y_1 y_3 y_4 v_0, \\
m(1, 0) = y_1 y_4 y_6 v_0+y_3 y_4 y_5 v_0, \\
m(0, 0) = y_1 y_3 y_5 y_6 v_0+y_1^{(2)} y_3 y_4 y_6 v_0+y_4 y_5 y_6 v_0.
\end{gather*}
This Weyl module is not ambiguous, in contrast to $\Delta(3,0)$.

\subsection{The morphism $\Delta(3,0) \to \Delta(4,0)$}
The character of the image of the morphism sending the generator of
$\Delta(3,0)$ onto $m(3,0)$ is
$\ell(3,0)+\ell(2,0)+2\ell(1,0)+2\ell(0,0)$.  The highest weights of
the composition factors of the socle series layers of its cokernel
are:
\[
\begin{array}{c|c}
\text{socle layer} & \text{highest weights}\\ \hline
4 & {(4, 0)}\\
3 & {(0, 2)}, \; {(0, 0)}\\
2 & {(2, 1)}, \; {(2, 0)}\\
1 & \phantom{.}{(0, 1)}, \; {(0, 0)} .
\end{array}
\]
The structure of this quotient suggests the possibility of a non-split
extension of $L(4,0)$ by $L(0,2)$. We will return to this question a
bit later.

\subsection{The morphism $\Delta(2,1) \to \Delta(4,0)$}
The character of the image of the morphism sending the generator of
$\Delta(2,1)$ to $m(2,1)$ is
$\ell(2,1)+\ell(2,0)+\ell(0,1)+\ell(1,0)+2\ell(0,0)$.  The highest
weights of the composition factors of the socle series layers of the
cokernel of this morphism are:
\[
\begin{array}{c|c}
\text{socle layer} & \text{highest weights}\\ \hline
6 & {(4, 0)}\\
5 & {(0, 0)}\\
4 & {(0, 2)}, \; {(2, 0)}\\
3 & {(3, 0)}\\
2 & {(1, 0)}\\
1 & \phantom{.}{(0, 0)} .
\end{array}
\]

\subsection{The morphism $\Delta(2,0) \to \Delta(4,0)$}
The character of the image of the morphism sending the generator of
$\Delta(2,0)$ onto $m(2,)$ is $\ell(2,0)+\ell(1,0)+\ell(0,0)$. This of
course determines the structure of this image. The highest weights of
the composition factors of the socle series layers of the cokernel of
this morphism are:
\[
\begin{array}{c|c}
\text{socle layer} & \text{highest weights}\\ \hline
6 & {(4, 0)}\\
5 & {(0, 0)}\\
4 & {(0, 2)}, \; {(2, 0)}\\
3 & {(3, 0)}\\
2 & {(2, 1)}, \; {(1, 0)}\\
1 & \phantom{.}{(0, 1)}, \; {(0, 0)}, \; {(0, 0)} .
\end{array}
\]
This quotient is an ambiguous module, since its socle contains a
direct sum of two copies of the trivial module $L(0,0) \cong \Bbbk$.

\subsection{The morphisms $\Delta(1,0) \to \Delta(4,0)$,
$\Delta(0,0) \to \Delta(4,0)$} There are also morphisms defined by
sending the generators of $\Delta(1,0)$ and $\Delta(0,0)$ respectively
to $m(1,0)$ and $m(0,0)$. The images of these morphisms are simple
modules.  We omit giving the socle series layers of their cokernels.

\subsection{Top factors of the radical}
Now we return to the issue of extensions of $L(4,0)$. We have
constructed quotients of $\Delta(4,0)$ of length two which verify that
\begin{equation}
  \Ext_G^1(L(4,0),L(0,2)) \ne 0 \quad\text{and}\quad
  \Ext_G^1(L(4,0),L(0,0)) \ne 0. 
\end{equation}
This information shows that $\Delta(4,0)$ is not rigid.

\section{\bf Structure of $\Delta(1,2)$}\noindent\label{ssDelta12}
The highest weights of the simple factors of the socle series layers
of this module are given by:
\[
\begin{array}{c|c}
\text{socle layer} & \text{highest weights}\\ \hline
12 & (1, 2)\\
11 & (2, 0)\\
10 & (0, 2),\; (0, 0)\\
9 & (4, 0)\\
8 & (0, 0)\\
7 & (0, 2),\; (2, 0)\\
6 & (3, 0)\\
5 & (2, 1),\; (1, 0)\\
4 & (0, 1),\; (0, 0)\\
3 & (2, 0)\\
2 & (0, 0)\\
1 & \phantom{.}(1, 0).
\end{array}
\]
The maximal vectors in $\Delta(1,2)$ are:
\begin{gather*}
m(1, 2) = v_0, \quad
m(4, 0) = y_2 v_0, \quad
m(2, 1) = y_1 y_2 v_0, \\
m(3, 0) = y_2 y_4 v_0, \quad 
m(2, 0) = y_1 y_2 y_3 y_4 v_0, \\
m(1, 0) = y_1 y_2 y_4 y_6 v_0+y_2 y_3 y_4 y_5 v_0+y_3 y_4 y_6 v_0 .
\end{gather*}
This module is not ambiguous.  It is also not rigid. There is an
extension of $L(1,2)$ by $L(0,2)$, in addition to the extension of
$L(1,2)$ by $L(2,0)$. To prove this, look at the quotient by the
submodule generated by the generator of the $11$th socle layer.

\subsection{The morphism $\Delta(4,0) \to \Delta(1,2)$}
The character of the image of the morphism sending the generator of
$\Delta(4,0)$ onto $m(4,0)$ is $\ell(4, 0) + \ell(2, 1) + \ell(0, 2) +
\ell(3, 0) + 2\ell(2, 0) + \ell(0, 1) + 2\ell(1, 0) + 3\ell(0, 0)$.
The highest weights of the composition factors of the socle series
layers of the cokernel of this morphism are:
\[
\begin{array}{c|c}
\text{socle layer} & \text{highest weights}\\ \hline
3 & (1, 2)\\
2 & (2, 0)\\
1 & \phantom{.}(0, 2),\; (0, 0).
\end{array}
\]

\subsection{The morphism $\Delta(2,1) \to \Delta(1,2)$}
The character of the image of the morphism sending the generator of
$\Delta(2,1)$ onto $m(2,1)$ is $\ell(2, 1) + \ell(2, 0) + \ell(0, 1) +
\ell(1, 0) + 2\ell(0, 0)$. The highest weights of the composition
factors of the socle series layers of the cokernel of this morphism
are:
\[
\begin{array}{c|c}
\text{socle layer} & \text{highest weights}\\ \hline
8 & (1, 2)\\
7 & (2, 0)\\
6 & (0, 2),\; (0, 0)\\
5 & (4, 0)\\
4 & (0, 0)\\
3 & (0, 2),\; (2, 0)\\
2 & (3, 0)\\
1 & \phantom{.}(1, 0).
\end{array}
\]

\subsection{The morphism $\Delta(3,0) \to \Delta(1,2)$}
The character of the image of the morphism sending the generator of
$\Delta(3,0)$ onto $m(3,0)$ is $\ell(3, 0) + \ell(2, 0) + 2\ell(1, 0)
+ 2\ell(0, 0)$.  The highest weights of the composition factors of the
socle series layers of the cokernel of this morphism are:
\[
\begin{array}{c|c}
\text{socle layer} & \text{highest weights}\\ \hline
7 & (1, 2)\\
6 & (2, 0)\\
5 & (0, 2),\; (0, 0)\\
4 & (4, 0)\\
3 & (0, 2),\; (0, 0)\\
2 & (2, 1),\; (2, 0)\\
1 & \phantom{.}(0, 1).
\end{array}
\]

\subsection{The morphism $\Delta(2,0) \to \Delta(1,2)$}%
\label{ss:20to12}  
The character of the image of the morphism sending the generator of
$\Delta(2,0)$ onto $m(2,0)$ is $\ell(2, 0) + \ell(1, 0) + \ell(0, 0)$.
The highest weights of the composition factors of the socle series
layers of the cokernel of this morphism are:
\[
\begin{array}{c|c}
\text{socle layer} & \text{highest weights}\\ \hline
9 & (1, 2)\\
8 & (2, 0)\\
7 & (0, 2),\; (0, 0)\\
6 & (4, 0)\\
5 & (0, 0)\\
4 & (0, 2),\; (2, 0)\\
3 & (3, 0)\\
2 & (2, 1),\; (1, 0)\\
1 & \phantom{.}(0, 1),\; (0, 0).
\end{array}
\]

\subsection{The morphism $\Delta(1,0) \to \Delta(1,2)$}
There is also a morphism mapping the generator of $\Delta(1,0)$ onto
$m(1,0)$, with simple image isomorphic to $L(1,0)$. We omit the socle
series layers for the cokernel of this morphism.


\section{\bf Structure of $\Delta(3,1)$}\noindent
The module $\Delta(3,1)$ has just two composition factors. The highest
weights of its socle layers are:
\[
\begin{array}{c|c}
\text{socle layer} & \text{highest weights}\\ \hline
2 & (3, 1)\\
1 & \phantom{.}(1, 1).
\end{array}
\]
The maximal vectors in $\Delta(3,1)$ are:
\begin{align*}
m(3, 1) &= v_0\\
m(1, 1) &= y_1 y_2 y_5 v_0+y_1 y_3 y_4 v_0
         + y_1 y_6 v_0+y_3 y_5 v_0+y_4^{(2)} v_0.
\end{align*}
The module is uniserial. We have $\St = L(1,1) = \Delta(1,1) =
T(1,1)$. Furthermore, $\Delta(3,1) \cong \St \otimes
\Delta(1,0)^{(1)}$. The tilting module $T(3,1)$ is uniserial length
three, with character $\ell(3,1)+2\ell(1,1)$, and head and socle
isomorphic to $L(1,1) = \St$.


\section{\bf Structure of $\Delta(0,3)$}\label{ss:Delta03}\noindent
The highest weights of the composition factors of the socle series
layers of $\Delta(0, 3)$ are:
\[
\begin{array}{c|c}
\text{socle layer} & \text{highest weights}\\ \hline
9 & (0, 3)\\
8 & (2, 0)\\
7 & (0, 0)\\
6 & (4, 0)\\
5 & (0, 0)\\
4 & (0, 2),\; (2, 0)\\
3 & (3, 0)\\
2 & (1, 0)\\
1 & \phantom{.}(0, 0).
\end{array}
\]
Its maximal vectors are:
\begin{gather*}
m(0, 3) = v_0, \quad 
m(4, 0) = y_2 y_3 v_0, \quad
m(0, 2) = y_2 y_5 v_0+y_3 y_4 v_0+y_6 v_0,\\
m(3, 0) = y_2 y_3 y_4 v_0+y_2 y_6 v_0, \quad
m(1, 0) = y_2 y_3^{(2)} y_4 y_5 v_0,\\
m(0, 0) = y_2 y_3 y_4 y_5 y_6 v_0 .
\end{gather*}
One can read off the image of the nonzero intertwining morphisms
$\Delta(\lambda) \to \Delta(0,3)$, where $\lambda$ is the weight of a
maximal vector, from the Alperin diagram displayed below. 

\subsection{Structure diagram of $\Delta(0,3)$}
The module $\Delta(0,3)$ is not ambiguous.  By constructing
submodules, we have determined that every primitive vector in the
socle series generates a submodule whose composition factors consist
of all those below it. Thus, this Weyl module is rigid and its
structure is determined by its Alperin diagram
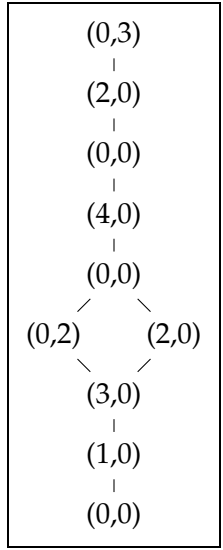
\[
\boxed{
\begin{tikzpicture}[scale = 0.8,thin, baseline={(0,-1ex/2)}]
  \node(a) at (0,0) {(0,3)};
  \node(b) at (0,-1) {(2,0)};
  \node(c) at (0,-2) {(0,0)};
  \node(d) at (0,-3) {(4,0)};
  \node(e) at (0,-4) {(0,0)};
  \node(f) at (-1,-5) {(0,2)};  \node(g) at (1,-5) {(2,0)};
  \node(h) at (0,-6) {(3,0)};
  \node(i) at (0,-7) {(1,0)};
  \node(j) at (0,-8) {(0,0)};
  \draw (a)--(b); 
  \draw (b)--(c);
  \draw (c)--(d);
  \draw (d)--(e);
  \draw (e)--(f); \draw (e)--(g);
  \draw (f)--(h); \draw (g)--(h);
  \draw (h)--(i);
  \draw (i)--(j);
\end{tikzpicture}
}
\]
The submodule generated by $m(4,0)$, which is the image of the
morphism sending the generator of $\Delta(4,0)$ onto $m(4,0)$,
provides another piece of information about $\Delta(4,0)$. This
analysis also shows that
\begin{equation}
  \Ext_G^1(L(0,3),L(2,0)) \cong \Bbbk
\end{equation}
and there are no other non-split extensions of $L(0,3)$.


\section{\bf Structure of $\Delta(5,0)$}\label{ss:Delta50}\noindent
The highest weights of the composition factors of the socle series
layers of $\Delta(5,0)$ are described by:
\[
\begin{array}{c|c}
\text{socle layer} & \text{highest weights}\\ \hline
11 & (5, 0)\\
10 & (1, 0)\\
9 & (1, 2),\; (3, 0)\\
8 & (2, 0)\\
7 & (0, 2),\; (0, 0)\\
6 & (4, 0)\\
5 & (0, 0)\\
4 & (4, 0),\; (0, 2),\; (2, 0)\\
3 & (0, 2),\; (3, 0),\; (0, 0)\\
2 & (2, 1),\; (2, 0),\; (1, 0)\\
1 & \phantom{.}(0, 1),\; (0, 0) .
\end{array}
\]
The maximal vectors in $\Delta(5,0)$ are listed below:
\begin{gather*}
m(5, 0) = v_0, \quad
m(1, 2) = y_1^{(2)} v_0, \quad
m_1(4, 0) = y_1 y_3 v_0, \quad
m_2(4, 0) = y_4 v_0, \\
m(2, 1) = y_1 y_4 v_0, \quad 
m(0, 2) = y_1^{(3)} y_3 v_0, \\
m(3, 0) = y_1 y_3 y_4 v_0+y_1 y_6 v_0+y_3 y_5 v_0, \\
m(2, 0) = y_1 y_4 y_6 v_0+y_3 y_4 y_5 v_0+y_4^{(3)} v_0, \\
m(0, 1) = y_1 y_3 y_4 y_5 v_0+y_1 y_4^{(3)} v_0, \\
m(1, 0) = y_1 y_3 y_4^{(3)} v_0+y_1 y_3 y_5 y_6 v_0+y_1 y_4^{(2)} y_6 v_0+y_3
 y_4^{(2)} y_5 v_0+y_4 y_5 y_6 v_0, \\
m(0, 0) = y_1 y_3 y_4 y_5 y_6 v_0 .
\end{gather*}
The module $\Delta(5,0)$ is ambiguous, as it has two linearly independent
maximal vectors of weight $(4,0)$.

\subsection{Top composition factors of the radical}
The vector $z = y_4^{(4)} v_0$ is a primitive vector of weight $(1,0)$
in $\Delta(5,0)$. It generates the top composition factor of the
ninth socle.  The quotient $\Delta(5,0)/\gen{z}$ by the submodule
generated by $z$ is a highest weight module with two composition
factors, namely $L(5,0)$ and $L(4,0)$. Thus we have
\begin{equation}
  \dim_\Bbbk\Ext_G^1(L(5,0),L(4,0)) \ne 0. 
\end{equation}
This shows immediately that $\Delta(5,0)$ is not rigid. We also have that
\begin{equation}
  \dim_\Bbbk\Ext_G^1(L(5,0),L(1,0)) \ne 0. 
\end{equation}
Indeed, a non-zero element of this space is obtained by constructing
the quotient $\Delta(5,0)$ by its ninth socle.

\subsection{The morphism $\Delta(1,2) \to \Delta(5,0)$}
The character of the image of the morphism sending the generator of
$\Delta(1,2)$ to $m(1,2)$ is $\ell(1, 2) + \ell(4, 0) + \ell(2, 1) +
2\ell(0, 2) + \ell(3, 0) + 2\ell(2, 0) + \ell(0, 1) + \ell(1, 0) +
3\ell(0, 0)$.  This image is isomorphic to the quotient of
$\Delta(1,2)$ by its third socle.  The highest weights of the composition factors of the socle series layers of the
cokernel of the morphism are:
\[
\begin{array}{c|c}
\text{socle layer} & \text{highest weights}\\ \hline
4 & (5, 0)\\
3 & (4, 0),\; (1, 0)\\
2 & (3, 0),\; (0, 0)\\
1 & \phantom{.}(0, 2),\; (2, 0).
\end{array}
\]

\subsection{The morphism $\varphi_1: \Delta(4,0) \to \Delta(5,0)$}
Let $\varphi_j$ be the morphism sending the generator of $\Delta(4,0)$
to the maximal vector $m_j(4,0)$, for $j=1,2$. The image of
$\varphi_1$ has character $\ell(4, 0) + \ell(0, 2) + \ell(3, 0) +
\ell(2, 0) + \ell(1, 0) + 2\ell(0, 0)$.  The socle layers of the
cokernel of $\varphi_1$ are:
\[
\begin{array}{c|c}
\text{socle layer} & \text{highest weights}\\ \hline
9 & (5, 0)\\
8 & (1, 0)\\
7 & (1, 2),\; (3, 0)\\
6 & (2, 0)\\
5 & (0, 2),\; (0, 0)\\
4 & (4, 0)\\
3 & (0, 2),\; (0, 0)\\
2 & (2, 1),\; (2, 0)\\
1 & \phantom{.}(0, 1).
\end{array}
\]

\subsection{The morphism $\varphi_2: \Delta(4,0) \to \Delta(5,0)$}
The image of $\varphi_2$ has character $\ell(4, 0) + \ell(2, 1) +
\ell(0, 2) + \ell(2, 0) + \ell(0, 1) + 2\ell(0, 0)$. The socle layers
of its cokernel are given by:
\[
\begin{array}{c|c}
\text{socle layer} & \text{highest weights}\\ \hline
10 & (5, 0)\\
9 & (1, 0)\\
8 & (1, 2),\; (3, 0)\\
7 & (2, 0)\\
6 & (0, 2),\; (0, 0)\\
5 & (4, 0)\\
4 & (0, 0)\\
3 & (0, 2),\; (2, 0)\\
2 & (3, 0)\\
1 & \phantom{.}(1, 0).
\end{array}
\]

\subsection{The morphism $\varphi_1+\varphi_2: \Delta(4,0) \to \Delta(5,0)$}
The character of the image of the morphism
$\varphi_1+\varphi_2$ is $\ell(4, 0) + \ell( 2, 1) + \ell(0, 2) +
\ell(3, 0) + \ell(2, 0) + \ell(0, 1) + \ell(1, 0) + 2\ell(0, 0)$. The
socle layers of its cokernel are:
\[
\begin{array}{c|c}
\text{socle layer} & \text{highest weights}\\ \hline
5 & (5, 0)\\
4 & (4, 0),\; (1, 0)\\
3 & (1, 2),\; (3, 0),\; (0, 0)\\
2 & (2, 0),\; (2, 0)\\
1 & \phantom{.}(0, 2),\; (0, 2),\; (0, 0).
\end{array}
\]

\subsection{A quotient of $\Delta(5,0)$}
The submodule of $\Delta(5,0)$ generated by both $m_1(4,0)$ and
$m_2(4,0)$ may also be of interest. Its character is $2\ell(4, 0) +
\ell(2, 1) + 2\ell(0, 2) + \ell(3, 0) + 2\ell(2, 0) + \ell(0, 1)\ +
\ell(1, 0) + 3\ell(0, 0)$. The corresponding quotient module
$\Delta(5,0)/\gen{m_1(4,0), m_2(4,0)}$ has socle layers of the form:
\[
\begin{array}{c|c}
\text{socle layer} & \text{highest weights}\\ \hline
5 & (5, 0)\\
4 & (1, 0)\\
3 & (1, 2),\; (3, 0)\\
2 & (2, 0)\\
1 & \phantom{.}(0, 2),\; (0, 0).
\end{array}
\]

\subsection{The morphism $\Delta(2,1) \to \Delta(5,0)$}
The character of the image of the morphism mapping the generator of
$\Delta(2,1)$ onto $m(2,1)$ is $\ell(2, 1) + \ell(0, 1) + \ell(0, 0)$.
The socle layers of the cokernel of this morphism are:
\[
\begin{array}{c|c}
\text{socle layer} & \text{highest weights}\\ \hline
10 & (5, 0)\\
9 & (1, 0)\\
8 & (1, 2),\; (3, 0)\\
7 & (2, 0)\\
6 & (0, 2),\; (0, 0)\\
5 & (4, 0)\\
4 & (0, 0)\\
3 & (4, 0),\; (0, 2),\; (2, 0)\\
2 & (3, 0),\; (0, 0)\\
1 & \phantom{.}(0, 2),\; (2, 0),\; (1, 0).
\end{array}
\]

\subsection{The morphism $\Delta(0,2) \to \Delta(5,0)$}
The character of the image of the morphism sending the generator of
$\Delta(0,2)$ to $m(0,2)$ is $\ell(0, 2) + \ell(3, 0) + \ell(1, 0) +
\ell(0, 0)$.  The layers of the socle series of the cokernel of this
morphism are:
\[
\begin{array}{c|c}
\text{socle layer} & \text{highest weights}\\ \hline
9 & (5, 0)\\
8 & (1, 0)\\
7 & (1, 2),\; (3, 0)\\
6 & (2, 0)\\
5 & (0, 2),\; (0, 0)\\
4 & (4, 0)\\
3 & (4, 0),\; (0, 2),\; (0, 0)\\
2 & (2, 1),\; (2, 0),\; (0, 0)\\
1 & \phantom{.}(2, 0),\; (0, 1).
\end{array}
\]

\subsection{The morphism $\Delta(3,0) \to \Delta(5,0)$}
The character of the image of the morphism sending the generator of
$\Delta(3,0)$ to $m(3,0)$ is $\ell(3, 0) + \ell(1, 0) + \ell(0,
0)$. The layers of the socle series of the cokernel of this morphism
are:
\[
\begin{array}{c|c}
\text{socle layer} & \text{highest weights}\\ \hline
9 & (5, 0)\\
8 & (1, 0)\\
7 & (1, 2),\; (3, 0)\\
6 & (2, 0)\\
5 & (0, 2),\; (0, 0)\\
4 & (4, 0)\\
3 & (4, 0),\; (0, 2),\; (0, 0)\\
2 & (2, 1),\; (2, 0),\; (0, 0)\\
1 & \phantom{.}(0, 2),\; (2, 0),\; (0, 1).
\end{array}
\]

\subsection{The morphism $\Delta(2,0) \to \Delta(5,0)$}
The character of the image of the morphism $\Delta(2,0) \to
\Delta(5,0)$ defined by sending the generator of $\Delta(2,0)$ onto
$m(2,0)$ is $\ell(2, 0) + \ell(0, 1)$. The socle layers of the
cokernel of this morphism are:
\[
\begin{array}{c|c}
\text{socle layer} & \text{highest weights}\\ \hline
11 & (5, 0)\\
10 & (1, 0)\\
9 & (1, 2),\; (3, 0)\\
8 & (2, 0)\\
7 & (0, 2),\; (0, 0)\\
6 & (4, 0)\\
5 & (0, 0)\\
4 & (4, 0),\; (0, 2),\; (2, 0)\\
3 & (0, 2),\; (3, 0),\; (0, 0)\\
2 & (2, 1),\; (1, 0)\\
1 & \phantom{.}(0, 0).
\end{array}
\]

\subsection{The morphism $\Delta(1,0) \to \Delta(5,0)$}
The character of the image of the morphism $\Delta(1,0) \to
\Delta(5,0)$ defined by mapping the generator of $\Delta(1,0)$ onto
$m(1,0)$ is $\ell(1, 0) + \ell(0, 0)$. This image is isomorphic to
$\Delta(1,0)$. The socle layers of the cokernel of this morphism 
are:
\[
\begin{array}{c|c}
\text{socle layer} & \text{highest weights}\\ \hline
9 & (5, 0)\\
8 & (1, 0)\\
7 & (1, 2),\; (3, 0)\\
6 & (2, 0)\\
5 & (0, 2),\; (0, 0)\\
4 & (4, 0),\; (4, 0)\\
3 & (0, 2),\; (0, 0),\; (0, 0)\\
2 & (2, 1),\; (0, 2),\; (2, 0),\; (2, 0)\\
1 & \phantom{.}(3, 0),\; (0, 1).
\end{array}
\]

\subsection{The morphisms $\Delta(0,1) \to \Delta(5,0)$ and
  $\Delta(0,0) \to \Delta(5,0)$} The images of the morphisms defined by
sending the generators of $\Delta(0,1)$ and $\Delta(0,0)$ to the
corresponding maximal vectors of the appropriate weight are simple
modules. We omit the socle series layers of their cokernels.


\section{\bf Structure of $\Delta(2,2)$}\label{ss:Delta22}\noindent
The highest weights of the composition factors of the socle series
layers of this Weyl module are:
\[
\begin{array}{c|c}
\text{socle layer} & \text{highest weights}\\ \hline
13 & (2, 2)\\
12 & (1, 2)\\
11 & (5, 0)\\
10 & (1, 0)\\
9 & (0, 3),\; (1, 2),\; (3, 0)\\
8 & (2, 0)\\
7 & (0, 2),\; (0, 0)\\
6 & (4, 0),\; (2, 0)\\
5 & (0, 0),\; (0, 0)\\
4 & (4, 0),\; (0, 2),\; (2, 0)\\
3 & (0, 2),\; (3, 0),\; (0, 0)\\
2 & (2, 1),\; (2, 0),\; (1, 0)\\
1 & \phantom{.}(0, 1),\; (0, 0).
\end{array}
\]
This module is ambiguous, with two linearly independent maximal
vectors of weight $(4,0)$. The complete list of maximal vectors in
$\Delta(2,2)$ is:
\begin{align*}
m(2, 2) &= v_0\\
m(0, 3) &= y_1 v_0\\
m(5, 0) &= y_2 v_0\\
m(1, 2) &= y_1^{(2)} y_2 v_0\\
m_1(4, 0) &= y_1 y_2 y_3 v_0\\
m_2(4, 0) &= y_2 y_4 v_0\\
m(2, 1) &= y_1 y_2 y_4 v_0\\
m(0, 2) &= y_1 y_2 y_5 v_0+y_1 y_3 y_4 v_0+y_1 y_6 v_0\\
m(3, 0) &= y_1 y_2 y_3 y_4 v_0+y_1 y_2 y_6 v_0+y_2 y_3 y_5 v_0+y_3 y_6 v_0\\
m(2, 0) &= y_1 y_2 y_4 y_6 v_0+y_2 y_3 y_4 y_5 v_0+y_2 y_4^{(3)} v_0+y_3 y_4 y\
_6 v_0\\
m(0, 1) &= y_1 y_2 y_3 y_4 y_5 v_0+y_1 y_2 y_4^{(3)} v_0+y_1 y_3 y_4 y_6 v_0\\
m(1, 0) &= y_1 y_2 y_3^{(2)} y_4 y_5 v_0+y_2 y_4 y_5 y_6 v_0+y_3^{(3)} y_4 y_5\
 v_0\\
m(0, 0) &= y_1 y_2 y_3 y_4 y_5 y_6 v_0 .
\end{align*}

\subsection{The morphism $\Delta(1,2) \to \Delta(2,2)$}
The character of the image of the morphism $\Delta(1,2) \to
\Delta(2,2)$ defined by sending the generator of $\Delta(1,2)$ to
$m(1,2)$ is $\ell(1, 2) + \ell(4, 0) + \ell(2, 1) + 2\ell(0, 2) +
\ell(3, 0) + 2\ell(2, 0) \ + \ell(0, 1) + \ell(1, 0) + 3\ell(0, 0)$.
This image is isomorphic to the quotient of $\Delta(1,2)$ by its
third socle; see \S\ref{ss:20to12}.  The socle layers of the cokernel
of this morphism are:
\[
\begin{array}{c|c}
\text{socle layer} & \text{highest weights}\\ \hline
7 & (2, 2)\\
6 & (0, 3),\; (1, 2)\\
5 & (2, 0)\\
4 & (5, 0),\; (0, 0)\\
3 & (4, 0),\; (1, 0)\\
2 & (3, 0),\; (0, 0)\\
1 & \phantom{.}(0, 2),\; (2, 0).
\end{array}
\]

\subsection{The morphism $\Delta(5,0) \to \Delta(2,2)$}
The character of the image of the morphism given by sending the
generator of $\Delta(5,0)$ to $m(5,0)$ is $\ell(5, 0) + \ell(1, 2) +
2\ell(4, 0) + \ell(2, 1) + 3\ell(0, 2) + 2\ell(3, 0)\ + 3\ell(2, 0) +
\ell(0, 1) + 2\ell(1, 0) + 4\ell(0, 0)$. The highest weights of the
composition factors of the socle series layers of the cokernel of this
morphism are:
\[
\begin{array}{c|c}
\text{socle layer} & \text{highest weights}\\ \hline
4 & (2, 2)\\
3 & (0, 3),\; (1, 2)\\
2 & (2, 0)\\
1 & \phantom{.}(0, 0).
\end{array}
\]

\subsection{The morphism $\Delta(0,3) \to \Delta(2,2)$}
The character of the image of the morphism sending the generator of
$\Delta(0,3)$ to $m(5,0)$ is $\ell(0, 3) + \ell(4, 0) + \ell(0, 2) +
\ell(3, 0) + 2\ell(2, 0) + \ell(1, 0) + 3\ell(0, 0)$. The socle series
layers of the cokernel of this morphism are:
\[
\begin{array}{c|c}
\text{socle layer} & \text{highest weights}\\ \hline
11 & (2, 2)\\
10 & (1, 2)\\
9 & (5, 0)\\
8 & (1, 0)\\
7 & (1, 2),\; (3, 0)\\
6 & (2, 0)\\
5 & (0, 2),\; (0, 0)\\
4 & (4, 0)\\
3 & (0, 2),\; (0, 0)\\
2 & (2, 1),\; (2, 0)\\
1 & \phantom{.}(0, 1).
\end{array}
\]

\subsection{The morphism $\varphi_1: \Delta(4,0) \to \Delta(2,2)$}
The character of the image of the morphism $\Delta(4,0) \to
\Delta(2,2)$ mapping the generator of $\Delta(4,0)$ onto $m_1(4,0)$ is
$\ell(4, 0) + \ell(0, 2) + \ell(3, 0) + \ell(2, 0) + \ell(1, 0) +
2\ell(0, 0)$. The highest weights of the composition factors of the socle series layers of the cokernel of this morphism
are:

\[
\begin{array}{c|c}
\text{socle layer} & \text{highest weights}\\ \hline
11 & (2, 2)\\
10 & (1, 2)\\
9 & (5, 0)\\
8 & (1, 0)\\
7 & (1, 2),\; (3, 0)\\
6 & (2, 0)\\
5 & (0, 2),\; (0, 0)\\
4 & (4, 0)\\
3 & (0, 3),\; (0, 2),\; (0, 0)\\
2 & (2, 1),\; (2, 0),\; (2, 0)\\
1 & \phantom{.}(0, 1),\; (0, 0) .
\end{array}
\]

\subsection{The morphism $\varphi_2:\Delta(4,0) \to \Delta(2,2)$}
The character of the image of the morphism $\Delta(4,0) \to
\Delta(2,2)$ sending the generator of $\Delta(4,0)$ onto $m_2(4,0)$ is
$ \ell(4, 0) + \ell(2, 1) + \ell(0, 2) + \ell(2, 0) + \ell(0, 1) +
2\ell(0, 0)$. The highest weights of the composition factors of the socle series layers of the cokernel of this morphism
are:
\[
\begin{array}{c|c}
\text{socle layer} & \text{highest weights}\\ \hline
12 & (2, 2)\\
11 & (1, 2)\\
10 & (5, 0)\\
9 & (1, 0)\\
8 & (0, 3),\; (1, 2),\; (3, 0)\\
7 & (2, 0)\\
6 & (0, 2),\; (0, 0)\\
5 & (4, 0)\\
4 & (0, 0)\\
3 & (0, 2),\; (2, 0),\; (2, 0)\\
2 & (3, 0),\; (0, 0)\\
1 & \phantom{.}(1, 0).
\end{array}
\]

\subsection{The morphism $\varphi_1+\varphi_2$}
The character of the image of the morphism $\Delta(4,0) \to
\Delta(2,2)$ mapping the generator of $\Delta(4,0)$ onto
$m_1(4,0)+m_2(4,0)$ is $\ell(4, 0) + \ell(2, 1) + \ell(0, 2) + \ell(3,
0) + \ell(2, 0) + \ell(0, 1) + \ \ell(1, 0) + 2\ell(0, 0)$. The
highest weights of the composition factors of the socle series layers
of the cokernel of this morphism are:
\[
\begin{array}{c|c}
\text{socle layer} & \text{highest weights}\\ \hline
8 & (2, 2)\\
7 & (0, 3),\; (1, 2)\\
6 & (2, 0)\\
5 & (5, 0),\; (0, 0)\\
4 & (4, 0),\; (1, 0)\\
3 & (1, 2),\; (3, 0),\; (0, 0)\\
2 & (2, 0),\; (2, 0)\\
1 & \phantom{.}(0, 2),\; (0, 2),\; (0, 0).
\end{array}
\]

\subsection{Another quotient of $\Delta(2,2)$}
The submodule $\gen{m_1(4,0), m_2(4,0)}$ generated by both maximal
vectors of weight $(4,0)$ has character $2\ell(4, 0) + \ell(2, 1) +
2\ell(0, 2) + \ell(3, 0) + 2\ell(2, 0) + \ell(0, 1)\ + \ell(1, 0) +
3\ell(0, 0)$. The highest weights of the composition factors of the
socle series layers of the corresponding quotient module $\Delta(2,2)/
\gen{m_1(4,0), m_2(4,0)}$ are:
\[
\begin{array}{c|c}
\text{socle layer} & \text{highest weights}\\ \hline
7 & (2, 2)\\
6 & (1, 2)\\
5 & (5, 0)\\
4 & (1, 0)\\
3 & (0, 3),\; (1, 2),\; (3, 0)\\
2 & (2, 0),\; (2, 0)\\
1 & \phantom{.}(0, 2),\; (0, 0),\; (0, 0).
\end{array}
\]

\subsection{The morphism $\Delta(2,1) \to \Delta(2,2)$}
The character of the image of the morphism mapping the generator of
$\Delta(2,1)$ onto $m(2,1)$ is $\ell(2, 1) + \ell(0, 1) + \ell(0,
0)$. The highest weights of the composition factors of the socle series layers of the cokernel of this morphism are:
\[
\begin{array}{c|c}
\text{socle layer} & \text{highest weights}\\ \hline
12 & (2, 2)\\
11 & (1, 2)\\
10 & (5, 0)\\
9 & (1, 0)\\
8 & (0, 3),\; (1, 2),\; (3, 0)\\
7 & (2, 0)\\
6 & (0, 2),\; (0, 0)\\
5 & (4, 0),\; (2, 0)\\
4 & (0, 0),\; (0, 0)\\
3 & (4, 0),\; (0, 2),\; (2, 0)\\
2 & (3, 0),\; (0, 0)\\
1 & \phantom{.}(0, 2),\; (2, 0),\; (1, 0).
\end{array}
\]

\subsection{The morphism $\Delta(0,2) \to \Delta(2,2)$}
The character of the image of the morphism sending the generator of
$\Delta(0,2)$ onto $m(0,2)$ is $\ell(0, 2) + \ell(3, 0) + \ell(1, 0) +
\ell(0, 0)$.  The highest weights of the composition factors of the socle series layers of the cokernel of this morphism
are:
\[
\begin{array}{c|c}
\text{socle layer} & \text{highest weights}\\ \hline
11 & (2, 2)\\
10 & (1, 2)\\
9 & (5, 0)\\
8 & (1, 0)\\
7 & (1, 2),\; (3, 0)\\
6 & (0, 3),\; (2, 0)\\
5 & (0, 2),\; (2, 0),\; (0, 0)\\
4 & (4, 0),\; (0, 0)\\
3 & (4, 0),\; (0, 2),\; (0, 0)\\
2 & (2, 1),\; (2, 0),\; (0, 0)\\
1 & \phantom{.}(2, 0),\; (0, 1).
\end{array}
\]

\subsection{The morphism $\Delta(3,0) \to \Delta(2,2)$}
The character of the image of the morphism mapping the generator of
$\Delta(3,0)$ onto $m(3,0)$ is $\ell(3, 0) + \ell(1, 0) + \ell(0,
0)$. The highest weights of the composition factors of the socle
series layers of the cokernel of this morphism are:
\[
\begin{array}{c|c}
\text{socle layer} & \text{highest weights}\\ \hline
11 & (2, 2)\\
10 & (1, 2)\\
9 & (5, 0)\\
8 & (1, 0)\\
7 & (0, 3),\; (1, 2),\; (3, 0)\\
6 & (2, 0),\; (2, 0)\\
5 & (0, 2),\; (0, 0),\; (0, 0)\\
4 & (4, 0),\; (4, 0)\\
3 & (0, 2),\; (0, 0),\; (0, 0)\\
2 & (2, 1),\; (2, 0),\; (2, 0)\\
1 & \phantom{.}(0, 2),\; (0, 1).
\end{array}
\]

\subsection{The morphism $\Delta(2,0) \to \Delta(2,2)$}
The character of the image of the morphism sending the generator of
$\Delta(2,0)$ to $m(2,0)$ is $\ell(2, 0) + \ell(0, 1)$. The socle
series layers of the cokernel of this morphism are:
\[
\begin{array}{c|c}
\text{socle layer} & \text{highest weights}\\ \hline
13 & (2, 2)\\
12 & (1, 2)\\
11 & (5, 0)\\
10 & (1, 0)\\
9 & (0, 3),\; (1, 2),\; (3, 0)\\
8 & (2, 0)\\
7 & (0, 2),\; (0, 0)\\
6 & (4, 0),\; (2, 0)\\
5 & (0, 0),\; (0, 0)\\
4 & (4, 0),\; (0, 2),\; (2, 0)\\
3 & (0, 2),\; (3, 0),\; (0, 0)\\
2 & (2, 1),\; (1, 0)\\
1 & \phantom{.}(0, 0).
\end{array}
\]

\subsection{The morphism $\Delta(1,0) \to \Delta(2,2)$}
The character of the image of the morphism mapping the generator of
$\Delta(1,0)$ onto $m(1,0)$ is $\ell(1, 0) + \ell(0, 0)$. The socle
series layers of the cokernel of this morphism are:
\[
\begin{array}{c|c}
\text{socle layer} & \text{highest weights}\\ \hline
11 & (2, 2)\\
10 & (1, 2)\\
9 & (5, 0)\\
8 & (1, 0)\\
7 & (0, 3),\; (1, 2),\; (3, 0)\\
6 & (2, 0),\; (2, 0)\\
5 & (0, 2),\; (0, 0),\; (0, 0)\\
4 & (4, 0),\; (4, 0)\\
3 & (0, 2),\; (0, 0),\; (0, 0)\\
2 & (2, 1),\; (0, 2),\; (2, 0),\; (2, 0)\\
1 & \phantom{.}(3, 0),\; (0, 1).
\end{array}
\]

\subsection{The morphisms from $\Delta(0,1)$ and $\Delta(0,0)$
  into $\Delta(2,2)$}
There are two additional morphisms $\Delta(\lambda) \to \Delta(2,2)$
defined by mapping the generator of $\Delta(\lambda)$ to $m(\lambda)$,
for $\lambda = (0,1)$ and $(0,0)$. The images of these two morphisms
are simple modules. We omit giving the socle series of their
cokernels.


\end{document}